\documentclass[ba,preprint]{imsart}

\RequirePackage{amsthm,amsmath,amsfonts,amssymb}
\RequirePackage{mathtools}
\RequirePackage{array}
\RequirePackage{tabularx}
\RequirePackage{booktabs}
\RequirePackage{graphicx}
\RequirePackage{placeins}
\RequirePackage[authoryear]{natbib}
\RequirePackage[colorlinks,citecolor=blue,urlcolor=blue,filecolor=blue,backref=page]{hyperref}

\pubyear{2026}
\volume{TBA}
\issue{TBA}
\firstpage{1}
\lastpage{15}

\startlocaldefs
\theoremstyle{plain}
\newtheorem{theorem}{Theorem}
\newtheorem{proposition}[theorem]{Proposition}

\theoremstyle{definition}

\endlocaldefs

\begin{document}

\begin{frontmatter}
\title{Exact Solutions to Unrestricted Generalized Variational Inference Problems}
\runtitle{Exact Solutions to Unrestricted GVI Problems}

\begin{aug}
\author[A]{\fnms{Jacob}~\snm{Westerhout}\ead[label=e1]{j.westerhout@uq.edu.au}}
\and
\author[B]{\fnms{Hien Duy}~\snm{Nguyen}\ead[label=e2]{h.nguyen5@latrobe.edu.au}\orcid{0000-0002-9958-432X}}
\address[A]{School of Mathematics and Physics,
The University of Queensland\printead[presep={,\ }]{e1}}

\address[B]{Department of Mathematical and Physical Sciences,
La Trobe University\printead[presep={,\ }]{e2}\newline
Institute of Mathematics for Industry, Kyushu University}
\runauthor{J. Westerhout and H. D. Nguyen}
\end{aug}

\begin{abstract}
We derive explicit formulas for the solutions of unrestricted generalized variational inference (GVI) problems, together with  necessary and sufficient conditions for the existence of a finite GVI solution. The results cover common $f$-divergences, separable Bregman divergences, and R\'enyi divergences, and hold for arbitrary measurable real-valued losses, without boundedness or integrability assumptions. In every case, the infinite-dimensional optimization over probability measures reduces to an explicit density indexed by a single scalar parameter, determined by a one-dimensional equation or optimization involving integrals against the prior or reference measure. For $f$-divergences and separable Bregman divergences, the formulas follow from a general characterization of each class; R\'enyi divergences fall outside both classes and are solved separately. Having the unrestricted posterior in explicit form separates the behavior of the divergence from any variational approximation, making both theoretical analysis and numerical computation more straightforward.
\end{abstract}



\end{frontmatter}

\section{Introduction}
\label{sec:intro}

Given a prior probability measure $P$, a loss function $\ell$, a divergence $\mathsf{D}$, and a regularization weight $\alpha>0$, we study the optimization problem
\begin{equation}
\label{eq:master-intro}
\underset{Q\in\mathcal{P}(\mathcal{Z})}{\mathrm{arg\,min}}
\left\{\int_{\mathcal{Z}} \ell(z)\,Q(\mathrm{d}z)+\alpha\,\mathsf{D}(Q\Vert P)\right\},
\end{equation}
where $Q$ ranges over the set $\mathcal{P}(\mathcal{Z})$ of all probability measures on $\mathcal{Z}$. When $\mathsf{D}$ is the Kullback--Leibler divergence, $\alpha=1$, and $\ell$ is a negative log-likelihood, the solution is the Bayesian posterior \citep{walker2006bayesian}. Retaining the Kullback--Leibler divergence while allowing a general loss yields the Gibbs posteriors of generalized Bayes \citep{bissiri2016general} and PAC-Bayes learning \citep{zhang2006information,alquier2024pacbayes}. Generalized variational inference (GVI) extends this framework by replacing the Kullback--Leibler penalty with other divergences, and typically restricts $Q$ to a tractable variational family \citep{knoblauch2022optimization,wild2022generalized}. Problem~\eqref{eq:master-intro} removes that restriction: its solution is the exact target approximated by any variational family, since restricting $Q$ can only increase the objective.

This paper derives explicit formulas for the solution of \eqref{eq:master-intro} under common divergences, together with exact criteria for when a finite solution exists. The reference point is the Kullback--Leibler case, in which a finite solution, if it exists, is the Gibbs posterior
\begin{equation}
\label{eq:gibbs-intro}
\frac{\mathrm{d}Q^\star}{\mathrm{d}P}(z)=\frac{\exp(-\ell(z)/\alpha)}{Z_\alpha},
\qquad
Z_\alpha=\int_{\mathcal{Z}}\exp(-\ell(z)/\alpha)\,P(\mathrm{d}z).
\end{equation}
Table~\ref{tab:solutions} summarizes the corresponding formulas and existence criteria for the other divergences. Most standard penalties belong to two classes, $f$-divergences and separable Bregman divergences, for which we solve \eqref{eq:master-intro} in full generality (Theorems~\ref{thm:f-main} and~\ref{thm:breg-main}). We also solve \eqref{eq:master-intro} for R\'enyi divergences, which fall outside both classes because the outer logarithm couples the density across $\mathcal{Z}$; consequently, R\'enyi penalties exhibit qualitatively different behavior.

Much of what is understood about Bayesian and generalized Bayesian posteriors rests on \eqref{eq:gibbs-intro}: having an explicit density makes it straightforward to analyze how the loss, prior, and $\alpha$ shape the posterior. For other divergences, 
the GVI posterior is often available only through an optimization over a variational family, confounding the behavior of the divergence with the choice of variational family. An explicit formula separates the two: one can inspect how a divergence treats the prior and extreme losses before choosing any variational family, and analyze theoretical properties such as robustness, support behavior, large-sample asymptotics, or variational approximation error directly from the density.

An explicit formula also simplifies computation. Each density in Table~\ref{tab:solutions} is determined up to a single scalar parameter, so an infinite-dimensional optimization over probability measures, which would otherwise require parametric restrictions or methods such as gradient flows on the space of measures \citep{wild2023rigorous}, reduces to a one-dimensional problem. The scalar plays the role of the normalizing constant $Z_\alpha$ in \eqref{eq:gibbs-intro}, but is not as simple to obtain: whereas $Z_\alpha$ factors out of the density and is a single integral, the scalar for other divergences generally enters the density nonlinearly. For the $f$-divergence and separable Bregman penalties this is a one-dimensional root-finding problem, and for R\'enyi penalties it is a one-dimensional optimization problem. In both cases integration against the prior (or the reference measure in the Bregman case) is still required, to determine the scalar and to compute posterior expectations.

Alongside the explicit formulas, we characterize \emph{precisely when} a finite GVI posterior exists for an arbitrary measurable real-valued loss, without assuming lower boundedness or prior integrability. These necessary and sufficient conditions separate candidate densities that merely integrate to one from genuine solutions with finite expected loss and finite divergence, and identify regimes in which the infimum is $-\infty$ or finite yet unattained.

Non-Kullback--Leibler penalties such as the density-power divergence \citep{basu1998robust} and R\'enyi divergences were introduced into GVI by \citet{knoblauch2022optimization} to improve robustness and control uncertainty, and problems mathematically equivalent to \eqref{eq:master-intro} arise across several related literatures. These include empirical risk minimization with $f$-divergence regularization \citep{daunas2024equivalence,daunas2025dual,daunas2026empirical}, online aggregation with $\varphi$-divergence penalties \citep{alquier2021nonexponentially}, Fenchel--Young generalizations of Bayesian updating \citep{sklaviadis2026fenchel}, sparse prediction maps \citep{martins2016sparsemax,martins2022sparse}, generalized Gibbs variational principles and distribution model risk \citep{breuer2016measuring,birrell2022fgamma}, and R\'enyi free energies on finite spaces \citep{giudice2021renyi}. 
The results cited above generally focus on a single divergence class, restrict attention to bounded, nonnegative, or prior-integrable losses, or assume that the optimal density exists and remains strictly positive. The results below collect explicit densities and exact finite-attainment criteria across $f$-, Bregman, and R\'enyi divergences for arbitrary measurable real-valued losses, including support-truncation and non-attainment regimes. Together, they give a common basis for analyzing and computing the solutions of these problems.

The remainder of the paper is organized as follows. Section~\ref{sec:solutions} presents Table~\ref{tab:solutions} and explains how to read and use the tabulated solutions. Section~\ref{sec:general} develops the general theory for $f$-divergences and separable Bregman divergences, and Section~\ref{sec:renyi} treats R\'enyi divergences. Section~\ref{sec:conclusion} outlines open numerical and statistical questions, and the Supplementary Material \citep{westerhout2026supplement} contains all proofs and supporting calculations; its sections, equations, and results are numbered with the prefix~S.

\section{Exact solutions}
\label{sec:solutions}

This section collects the exact solutions to the unrestricted GVI problem~\eqref{eq:master-intro} across common $f$-, separable Bregman, and R\'enyi divergences, together with necessary and sufficient conditions for a finite minimizer to exist. The results are summarized in Table~\ref{tab:solutions}; their general derivations are developed in Sections~\ref{sec:general} and~\ref{sec:renyi}. 

\paragraph{Standing assumptions.}
Throughout, $(\mathcal{Z},\mathfrak{Z})$ is a measurable space, $P\in\mathcal{P}(\mathcal{Z})$ is a prior probability measure, $\ell:\mathcal{Z}\to\mathbb{R}$ is a measurable real-valued loss, and $\alpha>0$. A probability measure $Q^\star\in\mathcal{P}(\mathcal{Z})$ achieves \emph{finite attainment} if it minimizes \eqref{eq:master-intro} with both terms separately finite, namely $\int_{\mathcal{Z}}|\ell|\,\mathrm{d}Q^\star<\infty$ and $\mathsf{D}(Q^\star\Vert P)<\infty$; cancellation between infinite terms is not permitted.

\paragraph{Base measures and densities.}
The divergences in Table~\ref{tab:solutions} differ in the base measure used to represent probability distributions as densities. Both $f$-divergences and R\'enyi divergences compare $Q$ directly against the prior probability measure $P$: we define them for $Q\ll P$ in terms of the Radon--Nikodym derivative $u=\mathrm{d}Q/\mathrm{d}P$, and set $\mathsf{D}(Q\Vert P)=\infty$ otherwise. Separable Bregman divergences instead compare densities with respect to a fixed $\sigma$-finite reference measure $\mu$ on $(\mathcal{Z},\mathfrak{Z})$ with $P\ll\mu$ (such as Lebesgue or counting measure). Writing $p=\mathrm{d}P/\mathrm{d}\mu$ and $q=\mathrm{d}Q/\mathrm{d}\mu$ for $Q\ll\mu$ (and setting $\mathsf{D}(Q\Vert P)=\infty$ otherwise), the Bregman rows express the penalty in terms of $p$ and $q$; for the Burg (Itakura--Saito) divergence we additionally assume $p>0$ $\mu$-almost everywhere. Accordingly, the solution column of Table~\ref{tab:solutions} reports the optimal density ratio $u^\star=\mathrm{d}Q^\star/\mathrm{d}P$ in the $f$-divergence and R\'enyi rows, and the optimal density $q^\star=\mathrm{d}Q^\star/\mathrm{d}\mu$ in the separable Bregman rows.

\begin{table}[!t]
\caption{Exact solutions and necessary and sufficient finite-attainment conditions for the unrestricted GVI problem~\eqref{eq:master-intro}. All notation is defined in Section~\ref{sec:solutions}, with the R\'enyi criterion~\eqref{eq:renyi-attainment-main} and constant $\rho_r$ given in Theorem~\ref{thm:renyi-main} and \eqref{eq:renyi-rho-main}.}
\label{tab:solutions}
\centering
\footnotesize
\renewcommand{\arraystretch}{1.8}
\setlength{\tabcolsep}{3pt}
\begin{tabularx}{\linewidth}{@{}>{\raggedright\arraybackslash}p{.155\linewidth}>{\raggedright\arraybackslash}p{.28\linewidth}>{\raggedright\arraybackslash}p{.315\linewidth}>{\raggedright\arraybackslash}X@{}}
\toprule
Divergence & $\mathsf{D}(Q\Vert P)$ & Solution & Finite attainment\newline iff\\
\midrule
\multicolumn{4}{@{}l}{\emph{$f$-divergences.}  The solution column gives $u^\star=\mathrm{d}Q^\star/\mathrm{d}P$.}\\
Kullback--Leibler
  & $\displaystyle\int u\log u\,\mathrm{d}P$
  & $e^{-\tilde\ell}=e^{-\ell/\alpha}/Z_\alpha$
  & $M_\alpha<\infty$\\
Cressie--Read, $\gamma>1$
  & $\displaystyle\int\frac{u^\gamma-\gamma u+\gamma-1}{\gamma(\gamma-1)}\,\mathrm{d}P$
  & $\bigl(1-(\gamma-1)\tilde\ell\,\bigr)_+^{1/(\gamma-1)}$
  & $X\in L^{\gamma/(\gamma-1)}(P)$\\
Pearson $\chi^2$
  & $\displaystyle\int(u-1)^2\,\mathrm{d}P$
  & $\bigl(1-\tilde\ell/2\bigr)_+$
  & $X\in L^2(P)$\\
Squared Hellinger
  & $\displaystyle\int(\sqrt{u}-1)^2\,\mathrm{d}P$
  & $\bigl(1+\tilde\ell\,\bigr)^{-2}$
  & normalizable\\
Reverse KL
  & $\displaystyle-\int\log u\,\mathrm{d}P$
  & $\bigl(1+\tilde\ell\,\bigr)^{-1}$
  & normalizable and $L_+<\infty$\\
Mixed KL, $a,b>0$
  & $\displaystyle\int(a\,u\log u-b\log u)\,\mathrm{d}P$
  & $\dfrac{b/a}{\mathrm{W}\bigl((b/a)\,e^{(b+\tilde\ell)/a}\bigr)}$
  & $M_{a\alpha}+L_+<\infty$\\
\midrule
\multicolumn{4}{@{}l}{\emph{Separable Bregman divergences.}  The solution column gives $q^\star=\mathrm{d}Q^\star/\mathrm{d}\mu$.}\\
Density power, $\beta>0$
  & $\begin{aligned}
     &\displaystyle\int\Bigl\{\frac{q^{\beta+1}-p^{\beta+1}}{\beta(\beta+1)}\\[-.4em]
     &\qquad-\frac{p^\beta(q-p)}{\beta}\Bigr\}\,\mathrm{d}\mu
     \end{aligned}$
  & $\bigl(p^\beta-\beta\tilde\ell\,\bigr)_+^{1/\beta}$
  & normalizable and $K_\beta<\infty$\\
Least squares
  & $\displaystyle\frac12\int(q-p)^2\,\mathrm{d}\mu$
  & $\bigl(p-\tilde\ell\,\bigr)_+$
  & normalizable and $K_1<\infty$\\
Burg/Itakura--Saito
  & $\displaystyle\int\Bigl\{\frac{q}{p}-\log\frac{q}{p}-1\Bigr\}\mathrm{d}\mu$
  & $\dfrac{p}{1+p\,\tilde\ell}$
  & normalizable and $K_{\mathrm{B}}<\infty$\\
\midrule
\multicolumn{4}{@{}l}{\emph{R\'enyi divergences.}  The solution column gives $u^\star=\mathrm{d}Q^\star/\mathrm{d}P$.}\\
R\'enyi, $r>1$
  & $\displaystyle\frac{1}{r-1}\log\int u^r\,\mathrm{d}P$
  & $\dfrac{(\lambda^\star-\ell)_+^{1/(r-1)}}{\int(\lambda^\star-\ell)_+^{1/(r-1)}\,\mathrm{d}P}$
  & \eqref{eq:renyi-attainment-main}\\[3pt]
R\'enyi, $0<r<1$
  & $\displaystyle\frac{1}{r-1}\log\int u^r\,\mathrm{d}P$
  & $\dfrac{(\ell-\tau^\star)^{-1/(1-r)}}{\int(\ell-\tau^\star)^{-1/(1-r)}\,\mathrm{d}P}$
  & $\ell_\star>-\infty$,\newline $\alpha r/(1-r)\ge\rho_r$\\
\bottomrule
\end{tabularx}
\end{table}

\paragraph{The scalar parameters $\lambda^\star$ and $\tau^\star$.}
Every solution in Table~\ref{tab:solutions} is determined up to a single scalar parameter. In the $f$-divergence and separable Bregman rows, this parameter is $\lambda^\star\in\mathbb{R}$, which enters through the shifted and scaled loss
\[
\tilde\ell=\frac{\ell+\lambda^\star}{\alpha}.
\]
It is the unique normalizing scalar that turns the displayed expression into a probability density with respect to its base measure ($P$ for $f$-divergences, $\mu$ for separable Bregman divergences)---that is, finite and nonnegative almost everywhere with total integral~$1$. We call a displayed density \emph{normalizable} when such a $\lambda^\star\in\mathbb{R}$ exists; in the reciprocal rows, this includes the positivity requirements $1+\tilde\ell>0$ ($P$-almost everywhere) for squared Hellinger and reverse Kullback--Leibler (KL), and $1+p\tilde\ell>0$ ($\mu$-almost everywhere) for Burg. Whenever finite attainment holds in these rows, both the optimizer and $\lambda^\star$ are unique. Only in the KL row does $\lambda^\star$ always have a closed form, $\lambda^\star=\alpha\log Z_\alpha$ with $Z_\alpha=\int_{\mathcal{Z}}e^{-\ell/\alpha}\,\mathrm{d}P$, recovering \eqref{eq:gibbs-intro}; in the remaining $f$-divergence and Bregman rows, $\lambda^\star$ is generally found numerically from the one-dimensional normalization equation (Section~\ref{subsec:mass-criterion}).

In the R\'enyi rows, by contrast, the displayed densities are already explicitly normalized, and the scalar parameter is obtained from a one-dimensional optimization problem rather than a normalization constraint. For $r>1$, $\lambda^\star\in\mathbb{R}$ is a global minimizer of the scalar profile $\Phi_r$ in \eqref{eq:renyi-profile-main} (Theorem~\ref{thm:renyi-main}); for $0<r<1$, $\tau^\star\in\mathbb{R}$ is the unique scalar parameter defined in Theorem~\ref{thm:renyi-subunit}.

\paragraph{Notation.}
The remaining notation in Table~\ref{tab:solutions} is as follows. We write $a_+=\max\{a,0\}$, $\ell^\pm=(\pm\ell)_+$, and $\log_+x=\max\{\log x,0\}$ for $x>0$, and denote by $\mathrm{W}$ the principal real branch of the Lambert $W$ function, the inverse of $w\mapsto we^w$ on $[0,\infty)$ \citep[Sec.~4.13]{olver2010nist}. To state the finite-attainment conditions compactly, we set
\begin{equation}
\label{eq:tail-notation-table}
\begin{gathered}
 X=\ell^-,\qquad
 M_c=\int_{\mathcal{Z}} X e^{X/c}\,\mathrm{d}P\quad(c>0),\\
 L_+=\int_{\mathcal{Z}}\log(1+\ell^+)\,\mathrm{d}P,\qquad
 \ell_\star=\operatorname*{ess\,inf}_P\ell,
\end{gathered}
\end{equation}
and, whenever the corresponding Bregman density $q^\star$ is normalizable,
\begin{equation}
\label{eq:breg-notation-table}
 K_\beta=\int_{\mathcal{Z}}\bigl|(q^\star)^{\beta+1}-p^{\beta+1}\bigr|\,\mathrm{d}\mu\quad(\beta>0),
 \qquad
 K_{\mathrm{B}}=\int_{\mathcal{Z}}\biggl\{\log_+\!\left(\frac{p}{q^\star}\right)+\left(\frac{q^\star}{p}-1\right)_+\biggr\}\,\mathrm{d}\mu.
\end{equation}
Finally, for the R\'enyi rows, the attainment criterion~\eqref{eq:renyi-attainment-main} for $r>1$ is given in Theorem~\ref{thm:renyi-main}, and the boundary constant $\rho_r$ for $0<r<1$ is defined in \eqref{eq:renyi-rho-main}.

\section[General theory for f-divergences and Bregman divergences]{General theory for $f$-divergences and Bregman divergences}
\label{sec:general}

This section solves the unrestricted GVI problem~\eqref{eq:master-intro} for two broad classes of penalties: $f$-divergences (Theorem~\ref{thm:f-main} in Section~\ref{sec:fdiv}) and separable Bregman divergences (Theorem~\ref{thm:breg-main} in Section~\ref{sec:bregman}). The $f$-divergence and separable Bregman rows of Table~\ref{tab:solutions} are the special cases obtained from particular generators. Each theorem gives a necessary and sufficient condition for finite attainment and identifies the minimizer up to a single scalar $\lambda$. Section~\ref{subsec:mass-criterion} then gives conditions under which this scalar exists and is unique, and describes how it is computed. All proofs are in the Supplementary Material.

Both theorems are stated in terms of the minimizer set
\[
 S_\varphi(y)=\operatorname*{arg\,min}_{t\ge0}\,\{\varphi(t)-yt\}
\]
for a convex function $\varphi:[0,\infty)\to(-\infty,\infty]$ and $y\in\mathbb{R}$. When $\varphi$ is strictly convex and differentiable on $(0,\infty)$, as in every $f$-divergence and separable Bregman row of Table~\ref{tab:solutions}, $S_\varphi(y)$ is either $\{(\varphi')^{-1}(y)\}$, $\{0\}$ (giving the positive parts in Table~\ref{tab:solutions}), or $\emptyset$.

\subsection[f-divergences]{$f$-divergences}
\label{sec:fdiv}

This subsection solves the unrestricted GVI problem~\eqref{eq:master-intro} when the penalty is an $f$-divergence of the form \eqref{eq:fdiv-def}. Let $f:[0,\infty)\to(-\infty,\infty]$ be convex, finite on $(0,\infty)$, and satisfy $f(1)=0$. The associated $f$-divergence is
\begin{equation}
\label{eq:fdiv-def}
\mathsf{D}_f(Q\Vert P)=
\begin{cases}
\displaystyle \int_{\mathcal{Z}} f\!\left(\frac{\mathrm{d} Q}{\mathrm{d} P}(z)\right)\,P(\mathrm{d}z), & Q\ll P,\\[0.5em]
\infty, & \text{otherwise.}
\end{cases}
\end{equation}

\begin{theorem}[$f$-divergence solution]
\label{thm:f-main}
Consider \eqref{eq:master-intro} with the $f$-divergence $\mathsf{D}_f$ in \eqref{eq:fdiv-def}. A measurable function $u:\mathcal{Z}\to[0,\infty)$ is the density ratio $\mathrm{d}Q^\star/\mathrm{d}P$ of a probability measure $Q^\star$ that achieves finite attainment if and only if, for some $\lambda\in\mathbb{R}$,
\begin{align}
 &u(z)\in S_f\!\left(-\frac{\ell(z)+\lambda}{\alpha}\right)
       \quad\text{for $P$-almost every $z$},
 \label{eq:u-inverse-main}\\
 &\int_{\mathcal{Z}} u\,\mathrm{d}P=1,\qquad
 \int_{\mathcal{Z}} |\ell|\,u\,\mathrm{d}P<\infty,\qquad
 \int_{\mathcal{Z}} f(u)\,\mathrm{d}P<\infty.
 \label{eq:f-normalization-main}
\end{align}
If some $Q_0\in\mathcal{P}(\mathcal{Z})$ satisfies $\int_{\mathcal{Z}}|\ell|\,\mathrm{d}Q_0<\infty$ and $\mathsf{D}_f(Q_0\Vert P)<\infty$, then the two finiteness conditions in \eqref{eq:f-normalization-main} may be replaced by the single condition $\int_{\mathcal{Z}}\ell^-u\,\mathrm{d}P<\infty$. If $f$ is strictly convex, then $Q^\star$ is unique. If $f$ is differentiable on $(0,\infty)$, then $\lambda$ is unique.
\end{theorem}

The proof is in Section~S2. The $f$-divergence rows of Table~\ref{tab:solutions} follow by evaluating $S_f$ for each generator and reducing \eqref{eq:f-normalization-main} to explicit conditions on the loss; these derivations are in Sections~S3 and~S4, and Section~S5 determines the value of the infimum when finite attainment fails.

\paragraph{Relation to existing results.}
The form \eqref{eq:u-inverse-main} of the optimizer is known in several literatures, in each case under restrictions on the loss or on the generator. Generalized Gibbs variational principles \citep{csiszar1995generalized}, and their use in measuring distribution model risk \citep{breuer2016measuring,birrell2022fgamma}, identify the optimizing density when the convex conjugate of the generator is differentiable, and guarantee that it exists for bounded losses \citep{birrell2022fgamma} or, for superlinear generators, for losses bounded below \citep{breuer2016measuring}. In online aggregation, \citet{alquier2021nonexponentially} obtains the same form, including the truncated Pearson and power updates, for nonnegative losses under differentiability, strict convexity, and growth conditions. In empirical risk minimization with $f$-divergence regularization \citep{daunas2024equivalence,daunas2025dual,daunas2026empirical}, the optimizer is characterized when the inverse-derivative expression is strictly positive, and \citet{daunas2026empirical} identify Pearson $\chi^2$ regimes in which this positivity fails. Theorem~\ref{thm:f-main} removes these restrictions. The loss is an arbitrary measurable real-valued function, which may be unbounded below and need not be integrable under the prior. 
The generator can be any convex function satisfying the assumptions preceding \eqref{eq:fdiv-def}, so solutions that vanish on part of the support of the prior are covered by the same statement. Finally, the characterization is necessary as well as sufficient, which is what yields the exact finite-attainment conditions in Table~\ref{tab:solutions}.

\paragraph{Reverse penalties.}
Theorem~\ref{thm:f-main} also covers the reverse penalty $\mathsf{D}_f(P\Vert Q)$, in which the two arguments of the divergence are exchanged. For mutually absolutely continuous $P$ and $Q$,
\[
  \mathsf{D}_f(P\Vert Q)=\mathsf{D}_{f^\diamond}(Q\Vert P),
  \qquad f^\diamond(t)=t f(1/t)\quad(t>0),\qquad f^\diamond(0)=\infty,
\]
and $f^\diamond$ is again convex and finite on $(0,\infty)$ with $f^\diamond(1)=0$, so the theorem applies with $f^\diamond$ in place of $f$. When $f$ is the Kullback--Leibler generator, $f^\diamond$ gives the reverse KL row of Table~\ref{tab:solutions} and $af+bf^\diamond$ gives the mixed KL row.

\subsection{Separable Bregman divergences}
\label{sec:bregman}

This subsection solves the unrestricted GVI problem~\eqref{eq:master-intro} when the penalty is a separable Bregman divergence, generalizing the Bregman rows of Table~\ref{tab:solutions}. Recall from Section~\ref{sec:solutions} that these penalties compare densities with respect to a fixed $\sigma$-finite reference measure $\mu$ with $P\ll\mu$, and that $p=\mathrm{d}P/\mathrm{d}\mu$. Let $\psi:[0,\infty)\to(-\infty,\infty]$ be convex and finite on $(0,\infty)$, and let $s:\mathcal{Z}\to\mathbb{R}$ be a measurable subgradient selection with $s(z)\in\partial\psi(p(z))$ for $\mu$-almost every $z$, so that $s=\psi'(p)$ wherever $\psi$ is differentiable at $p$. The separable Bregman divergence generated by $\psi$ is
\begin{equation}
\label{eq:breg-def-main}
\mathsf{D}_{\psi}(Q\Vert P)
=
\begin{cases}
\displaystyle\int_{\mathcal{Z}} B(z,q(z))\,\mu(\mathrm{d}z),
  & Q\ll\mu,\quad q=\mathrm{d}Q/\mathrm{d}\mu,\\[0.6em]
\infty, & \text{otherwise,}
\end{cases}
\end{equation}
where $B(z,t)=\psi(t)-\psi(p(z))-s(z)\bigl(t-p(z)\bigr)\ge0$.

\begin{theorem}[Separable Bregman solution]
\label{thm:breg-main}
Consider \eqref{eq:master-intro} with the separable Bregman divergence $\mathsf{D}_\psi$ in \eqref{eq:breg-def-main}. A measurable function $q:\mathcal{Z}\to[0,\infty)$ is the density $\mathrm{d}Q^\star/\mathrm{d}\mu$ of a probability measure $Q^\star$ that achieves finite attainment if and only if, for some $\lambda\in\mathbb{R}$,
\begin{align}
 &q(z)\in S_\psi\!\left(s(z)-\frac{\ell(z)+\lambda}{\alpha}\right)
 \quad\text{for $\mu$-almost every $z$},
 \label{eq:q-lambda-main}\\
 &\int_{\mathcal{Z}} q\,\mathrm{d}\mu=1,\qquad
 \int_{\mathcal{Z}} |\ell|\,q\,\mathrm{d}\mu<\infty,\qquad
 \int_{\mathcal{Z}} B(\cdot,q)\,\mathrm{d}\mu<\infty.
 \label{eq:breg-finite-main}
\end{align}
If some $Q_0\in\mathcal{P}(\mathcal{Z})$ satisfies $\int_{\mathcal{Z}}|\ell|\,\mathrm{d}Q_0<\infty$ and $\mathsf{D}_\psi(Q_0\Vert P)<\infty$, then the two finiteness conditions in \eqref{eq:breg-finite-main} may be replaced by the single condition $\int_{\mathcal{Z}}\ell^-q\,\mathrm{d}\mu<\infty$. If $\psi$ is strictly convex, then $Q^\star$ is unique. If $\psi$ is differentiable on $(0,\infty)$, then $\lambda$ is unique.
\end{theorem}

The proof is in Section~S6. The Bregman rows of Table~\ref{tab:solutions} follow by evaluating $S_\psi$ for the density power and Burg generators and reducing the finiteness conditions in \eqref{eq:breg-finite-main} to the integrals $K_\beta$ and $K_{\mathrm{B}}$ in \eqref{eq:breg-notation-table}; these derivations are in Section~S7, which also gives equivalent conditions stated directly in terms of the loss.

\paragraph{Relation to existing results.}
Bregman-regularized versions of \eqref{eq:master-intro} arise in several related lines of work. The closest is the Fenchel--Young generalization of Bayesian updating of \citet{sklaviadis2026fenchel}, in which the prior enters through a potential. When that potential is a subgradient of the generator at the prior density, their discrepancy is the Bregman divergence \eqref{eq:breg-def-main}, and they characterize the posterior, under integrability conditions on the loss and the potential, when the infimum is attained. Second, \eqref{eq:q-lambda-main} is the optimality condition of a Bregman proximal, or mirror-descent, step \citep{beck2003mirror} that starts at the prior density and moves in the direction $-\ell$ with step size $1/\alpha$; Remark~S7 in Section~S6 makes this precise. Third, on a finite space with counting measure, the least-squares solution is the Euclidean projection of $p-\ell/\alpha$ onto the probability simplex, that is, $q^\star=\operatorname{sparsemax}(p-\ell/\alpha)$ \citep{martins2016sparsemax}, and truncated power densities of the kind in the density power row appear in regularized prediction on general measure spaces \citep{martins2022sparse}. Bregman integrands are also admitted by the fixed-penalty theory of \citet{breuer2016measuring}. Theorem~\ref{thm:breg-main} gives a common characterization of these solutions that holds for arbitrary measurable real-valued losses and any $\sigma$-finite reference measure dominating $P$. Because it is necessary as well as sufficient, it also decides when no finite solution exists, which is what yields the finite-attainment conditions in the Bregman rows of Table~\ref{tab:solutions}.

\subsection{Guaranteeing normalization}
\label{subsec:mass-criterion}

Theorems~\ref{thm:f-main} and~\ref{thm:breg-main} characterize solutions through a scalar $\lambda$ that must make the candidate density integrate to one, together with finiteness conditions. This subsection gives sufficient conditions under which a unique normalizing scalar $\lambda^\star$ exists, and notes how it is computed.

Suppose that the generator is strictly convex and differentiable on $(0,\infty)$, as in every $f$-divergence and separable Bregman row of Table~\ref{tab:solutions}. Then \eqref{eq:u-inverse-main} and \eqref{eq:q-lambda-main} admit at most one candidate for each $\lambda$. Write $v_\lambda$ for this candidate and $\nu$ for its base measure: in the $f$-divergence case, $\nu=P$ and $v_\lambda(z)$ is the unique element of $S_f(-(\ell(z)+\lambda)/\alpha)$; in the Bregman case, $\nu=\mu$ and $v_\lambda(z)$ is the unique element of $S_\psi(s(z)-(\ell(z)+\lambda)/\alpha)$. Whenever $v_\lambda$ is defined $\nu$-almost everywhere, its total mass is
\[
G(\lambda)=\int_{\mathcal{Z}}v_\lambda(z)\,\nu(\mathrm{d}z)\in[0,\infty],
\]
and $\lambda$ normalizes the candidate when $G(\lambda)=1$.

\begin{proposition}[Existence and uniqueness of the normalizing scalar]
\label{prop:normalizer-main}
Let the generator be strictly convex and differentiable on $(0,\infty)$, and let $\Lambda=(\lambda_-,\lambda_+)\subseteq\mathbb{R}$ be an open interval, possibly unbounded, such that $v_\lambda$ is defined $\nu$-almost everywhere and $G(\lambda)<\infty$ for every $\lambda\in\Lambda$. Then $G$ is continuous and nonincreasing on $\Lambda$, and strictly decreasing wherever it is positive. Consequently, if
\[
\lim_{\lambda\uparrow\lambda_+}G(\lambda)<1<
\lim_{\lambda\downarrow\lambda_-}G(\lambda),
\]
then there is a unique $\lambda^\star\in\Lambda$ with $G(\lambda^\star)=1$.
\end{proposition}

The proof is in Section~S8. Proposition~\ref{prop:normalizer-main} reduces normalization to two checks. The first is that $G$ is finite, which is not automatic: it can fail when the loss has a heavy lower tail, or when the reference measure $\mu$ has infinite mass. The second is the pair of endpoint limits. When the limit at $\lambda_-$ does not exceed one, the candidate need not be normalizable.  Normalization is also only part of finite attainment: by Theorems~\ref{thm:f-main} and~\ref{thm:breg-main}, the normalized candidate $v_{\lambda^\star}$ solves \eqref{eq:master-intro} if and only if it also satisfies the finiteness conditions in \eqref{eq:f-normalization-main} or \eqref{eq:breg-finite-main}.

Under the conditions of Proposition~\ref{prop:normalizer-main}, $\lambda^\star$ is the root of a continuous monotone function of one variable, so it can be computed by any standard one-dimensional root-finding method, started from $\lambda_L<\lambda_U$ in $\Lambda$ with $G(\lambda_L)>1>G(\lambda_U)$. Each evaluation of $G$ requires one integral against $P$ or $\mu$, so the cost of computing $\lambda^\star$ is that of a short sequence of such integrals.

\section{R\'enyi divergences}
\label{sec:renyi}

This section solves the unrestricted GVI problem~\eqref{eq:master-intro} for R\'enyi divergences.  For $r>0$ with $r\ne1$, and for $Q\ll P$ with $u=\mathrm{d}Q/\mathrm{d}P$, the R\'enyi divergence of order $r$ is
\begin{equation}
\label{eq:renyi-def-main}
  \mathsf{D}_{\mathrm{R}}^{(r)}(Q\Vert P)
  =\frac{1}{r-1}\log\int_{\mathcal{Z}}u(z)^r\,P(\mathrm{d}z).
\end{equation}
As the logarithm is applied to an integral over $\mathcal Z$, rather than to the density pointwise, the separable theory of Section~\ref{sec:general} does not apply. Theorem~\ref{thm:renyi-main} in Section~\ref{subsec:renyi-above} ($r>1$) and Theorem~\ref{thm:renyi-subunit} in Section~\ref{subsec:renyi-attainment} ($0<r<1$) nevertheless solve \eqref{eq:master-intro} exactly: each gives the necessary and sufficient condition for finite attainment listed in Table~\ref{tab:solutions} and determines the scalar in the tabulated density. That scalar plays a different role from the one in Section~\ref{subsec:mass-criterion}. The R\'enyi densities in Table~\ref{tab:solutions} integrate to one for every value of the scalar, so normalization cannot determine it; instead, the scalar solves a one-dimensional optimization problem. For $r>1$ this problem is a global minimization of a function that need not be convex, and its minimizer need not be unique. For $0<r<1$ it is concave and no harder than the root-finding of Section~\ref{subsec:mass-criterion}. All proofs are in Sections~S9 and~S11.

\subsection{Orders greater than one}
\label{subsec:renyi-above}

This subsection solves the unrestricted GVI problem~\eqref{eq:master-intro} for R\'enyi divergences of order $r>1$. Put $h=r/(r-1)$, so that the exponent in the corresponding row of Table~\ref{tab:solutions} is $1/(r-1)=h-1$. For $p>0$ and $\lambda\in\mathbb{R}$, define
\begin{equation}
\label{eq:A-def-main}
  A_p(\lambda)=\int_{\mathcal{Z}}(\lambda-\ell)_+^p\,\mathrm{d}P\in[0,\infty],
\end{equation}
so that the tabulated density with scalar $\lambda$ is $(\lambda-\ell)_+^{h-1}/A_{h-1}(\lambda)$. Both the attainment condition and the optimal scalar are expressed through the profile
\begin{equation}
\label{eq:renyi-profile-main}
  \Phi_r(\lambda)=\lambda-\alpha\log A_h(\lambda),
  \qquad \lambda>\ell_\star,
\end{equation}
which is defined whenever $A_h$ is finite, and through the constant
\begin{equation}
\label{eq:renyi-tail-constants-main}
  T_{r,\alpha}=\limsup_{\lambda\to-\infty}e^{-\lambda/\alpha}A_h(\lambda)\in[0,\infty],
\end{equation}
which depends only on the lower tail of the loss.

\begin{theorem}[R\'enyi solution for $r>1$]
\label{thm:renyi-main}
Consider \eqref{eq:master-intro} with the R\'enyi divergence \eqref{eq:renyi-def-main} of order $r>1$. Finite attainment holds if and only if
\begin{equation}
\label{eq:renyi-attainment-main}
  T_{r,\alpha}<\infty
  \quad\text{and}\quad
  e^{-b/\alpha}A_h(b)\ge T_{r,\alpha}
  \quad\text{for some }b>\ell_\star.
\end{equation}
In that case, $\Phi_r$ is finite and continuously differentiable on $(\ell_\star,\infty)$ and attains its minimum, and a probability measure $Q^\star$ achieves finite attainment if and only if
\[
  \frac{\mathrm{d}Q^\star}{\mathrm{d}P}
  =\frac{(\lambda^\star-\ell)_+^{h-1}}{A_{h-1}(\lambda^\star)}
  \qquad\text{for some global minimizer $\lambda^\star$ of $\Phi_r$.}
\]
Every such $\lambda^\star$ is a stationary point of $\Phi_r$, that is, a solution of the threshold equation
\begin{equation}
\label{eq:theta-main}
  \frac{A_h(\lambda)}{A_{h-1}(\lambda)}=\frac{\alpha r}{r-1}.
\end{equation}
\end{theorem}

The proof is in Section~S9. The optimal density is a truncated power 
of the gap between the threshold and the loss: it places zero mass on $\{\ell\ge\lambda^\star\}$ and weights smaller losses proportionally to $(\lambda^\star-\ell)^{1/(r-1)}$. By the stationarity condition \eqref{eq:theta-main}, the truncation threshold sits exactly $\alpha r/(r-1)$ above the posterior expected loss,
\[
  \lambda^\star=\int_{\mathcal{Z}}\ell\,\mathrm{d}Q^\star+\frac{\alpha r}{r-1}.
\]

\paragraph{When a finite solution exists.}
In most standard settings, finite attainment holds automatically. In particular, \eqref{eq:renyi-attainment-main} is satisfied whenever the loss is bounded below, and more generally whenever the Gibbs normalizing constant $Z_\alpha=\int_{\mathcal{Z}}e^{-\ell/\alpha}\,\mathrm{d}P$ is finite. No integrability condition on the upper tail $\ell^+$ is needed, because high losses are truncated away.

The general criterion \eqref{eq:renyi-attainment-main} is only needed when $\ell$ is unbounded below with a heavy lower tail. Solving \eqref{eq:master-intro} reduces to minimizing the one-dimensional profile $\Phi_r(\lambda)$ over $\lambda>\ell_\star$. Since $\Phi_r(\lambda)\to\infty$ as $\lambda\to\infty$, 
attainment can fail when the infimum is approached along thresholds tending to $-\infty$ and is not achieved at any finite threshold.
Because $e^{-\lambda/\alpha}A_h(\lambda)=\exp\{-\Phi_r(\lambda)/\alpha\}$, minimizing $\Phi_r$ is equivalent to maximizing $\lambda\mapsto e^{-\lambda/\alpha}A_h(\lambda)$, and $T_{r,\alpha}$ is its limiting upper value as $\lambda\to-\infty$. Thus $T_{r,\alpha}=\infty$ means the infimum in \eqref{eq:master-intro} is $-\infty$, while when $T_{r,\alpha}<\infty$ the second inequality in \eqref{eq:renyi-attainment-main}  requires that this function at some finite threshold $b$ be at least as large as its $-\infty$ limit; otherwise the infimum is finite and unattained.

As a function of the regularization weight $\alpha$, this behavior separates at a single critical threshold $\alpha_c\in[0,\infty]$ determined by the exponential decay rate of the lower tail $\ell^-$ (independent of $r>1$): finite attainment holds for every $\alpha>\alpha_c$ (where $T_{r,\alpha}=0$), the infimum is $-\infty$ for every $\alpha<\alpha_c$, and only at the boundary $\alpha=\alpha_c$ does the comparison in \eqref{eq:renyi-attainment-main} come into play. Full details are given in Proposition~S11, Corollaries~S12--S14, and Proposition~S15 in Section~S9.

\paragraph{Finding $\lambda^\star$.}
Unlike the normalizing scalar in Section~\ref{subsec:mass-criterion}, $\lambda^\star$ cannot in general be found by solving the stationarity equation~\eqref{eq:theta-main} with a root-finder. Because the profile $\Phi_r$ need not be convex, \eqref{eq:theta-main} is necessary but not sufficient: it can have spurious roots corresponding to local maxima or suboptimal local minima, and even the global minimizer need not be unique (in which case \eqref{eq:master-intro} has multiple solutions).

This non-convexity already appears on a two-point space. Take $r=2$, $\alpha=4$, losses $\ell(1)=0$ and $\ell(2)=10$, and prior weights $P(\{1\})=0.1$ and $P(\{2\})=0.9$. Here \eqref{eq:theta-main} has three roots: $\lambda=8$ (a local minimum with $\Phi_2\approx 0.575$), $\lambda=13-\sqrt{7}$ (a local maximum with $\Phi_2\approx 0.823$), and $\lambda=13+\sqrt{7}$ (the global minimum with $\Phi_2\approx -0.248$). A local solver can therefore get trapped at $\lambda=8$, giving a density supported only on the first point, whereas the true solution $\lambda^\star=13+\sqrt{7}$ puts mass on both points.

Computing $\lambda^\star$ therefore requires minimizing $\Phi_r$ globally over $(\ell_\star,\infty)$, with each evaluation costing one integral against $P$. Whenever $\ell_\star>-\infty$ or $Z_\alpha<\infty$, we have $\Phi_r(\lambda)\to\infty$ at both ends of $(\ell_\star,\infty)$, so any initial guess $\lambda_0>\ell_\star$ confines every global minimizer to a bounded sublevel set $\{\Phi_r\le\Phi_r(\lambda_0)\}$ (Proposition~S11 in Section~S9 also gives an explicit bounding interval when $\ell_\star>-\infty$). On a finite space, the candidate supports $\{\ell<\lambda\}$ are nested and can be checked directly; Section~S10 implements this for $r=2$ and a uniform prior on $K$ points at cost $O(K\log K)$. Further details---including a sufficient condition for uniqueness, an example with non-unique global minimizers, and an example where \eqref{eq:theta-main} has roots even when no solution exists---are collected in Section~S9.

\subsection{Orders between zero and one}
\label{subsec:renyi-attainment}

This subsection solves the unrestricted GVI problem~\eqref{eq:master-intro} for R\'enyi divergences of order $0<r<1$. Put $k=r/(1-r)$, so that the exponent in the corresponding row of Table~\ref{tab:solutions} is $-1/(1-r)=-(k+1)$. For these orders the general R\'enyi divergence can be finite for measures with a component singular to $P$ \citep[Sec.~II-A]{vanerven2014renyi};
the result below therefore concerns the dominated formulation specified in Section~\ref{sec:solutions}.
When $\ell_\star>-\infty$, define, for $p>0$ and $\tau\le\ell_\star$,
\begin{equation}
\label{eq:renyi-rho-main}
  C_p(\tau)=\int_{\mathcal{Z}}(\ell-\tau)^{-p}\,\mathrm{d}P\in(0,\infty],
  \qquad
  \rho_r=\lim_{\tau\uparrow\ell_\star}\frac{C_k(\tau)}{C_{k+1}(\tau)}\in[0,\infty),
\end{equation}
with the convention $0^{-p}=\infty$. Each $C_p(\tau)$ is finite for $\tau<\ell_\star$, and the limit exists because the ratio $C_k/C_{k+1}$ is continuous and strictly decreasing on $(-\infty,\ell_\star)$.

\begin{theorem}[R\'enyi solution for $0<r<1$]
\label{thm:renyi-subunit}
Consider \eqref{eq:master-intro} with the R\'enyi divergence \eqref{eq:renyi-def-main} of order $0<r<1$. Finite attainment holds if and only if
\begin{equation}
\label{eq:renyi-subunit-attainment}
  \ell_\star>-\infty
  \quad\text{and}\quad
  \frac{\alpha r}{1-r}\ge\rho_r.
\end{equation}
In that case the minimizer $Q^\star$ is unique, and
\[
  \frac{\mathrm{d}Q^\star}{\mathrm{d}P}
  =\frac{(\ell-\tau^\star)^{-1/(1-r)}}{C_{k+1}(\tau^\star)},
\]
where $\tau^\star\le\ell_\star$ is determined as follows.
\begin{enumerate}
\item[\textup{(i)}] If $\alpha r/(1-r)>\rho_r$, then $\tau^\star<\ell_\star$ is the unique solution of
\begin{equation}
\label{eq:renyi-ratio-main}
  \frac{C_k(\tau)}{C_{k+1}(\tau)}=\frac{\alpha r}{1-r}.
\end{equation}
\item[\textup{(ii)}] If $\alpha r/(1-r)=\rho_r$, then $\tau^\star=\ell_\star$; in this case $C_{k+1}(\ell_\star)<\infty$ and $P(\ell=\ell_\star)=0$.
\end{enumerate}
\end{theorem}

The proof is in Section~S11. The behavior for $0<r<1$ contrasts sharply with $r>1$. Because $\tau^\star\le\ell_\star$, the optimal density is strictly positive: no losses are truncated, and the scalar $\tau^\star$ sits $\alpha r/(1-r)$ below the posterior expected loss,
\[
  \tau^\star=\int_{\mathcal{Z}}\ell\,\mathrm{d}Q^\star-\frac{\alpha r}{1-r}.
\]
Moreover, the objective in \eqref{eq:master-intro} is strictly convex for $0<r<1$, so whenever a finite solution exists it is unique.

\paragraph{When a finite solution exists.}
Two mechanisms can prevent finite attainment. First, the loss must be essentially bounded below: if $\ell_\star=-\infty$, the infimum in \eqref{eq:master-intro} is $-\infty$ for every $\alpha>0$, so unlike the $r>1$ case, no regularization weight can compensate for an unbounded lower tail.

Second, even when $\ell$ is bounded below (which guarantees a finite infimum), attainment can fail if the prior puts too little mass near $\ell_\star$ and $\alpha$ is too small, in which case minimizing sequences assign a fixed positive amount of mass to sets on which the loss approaches $\ell_\star$ without converging to a valid density. The constant $\rho_r$ captures this threshold and can be evaluated directly without taking a limit: whenever $C_{k+1}(\ell_\star)=\infty$, we have $\rho_r=0$ and a unique solution exists for every $\alpha>0$. This includes every prior with $P(\ell=\ell_\star)>0$, and in particular every finite space with positive prior weights. Otherwise $C_{k+1}(\ell_\star)<\infty$ and $\rho_r=C_k(\ell_\star)/C_{k+1}(\ell_\star)>0$; a finite solution then exists if and only if $\alpha\ge(1-r)\rho_r/r$, while for smaller $\alpha$ the infimum is finite and unattained (Proposition~S16 in Section~S11).

For a concrete example of this boundary regime, take $P(\mathrm{d}z)=3z^2\,\mathrm{d}z$ on $[0,1]$, $\ell(z)=z$, and $r=1/2$. Then $\ell_\star=0$ and $k=1$, with $C_1(0)=3/2$ and $C_2(0)=3$, so $\rho_{1/2}=1/2$. No solution exists for $0<\alpha<1/2$, while at $\alpha=1/2$, case~(ii) applies with $\tau^\star=0$ and $\mathrm{d}Q^\star/\mathrm{d}P=1/(3z^2)$, so $Q^\star$ is uniform on $[0,1]$.

\paragraph{Finding $\tau^\star$.}
Unlike the $r>1$ case, computing the scalar is straightforward. Whenever \eqref{eq:renyi-subunit-attainment} holds, $\tau^\star$ is the unique maximizer over $(-\infty,\ell_\star]$ of the scalar profile
\begin{equation}
\label{eq:renyi-subunit-profile-main}
  \Psi_r(\tau)=\tau-\alpha\log C_k(\tau),
\end{equation}
which has the same algebraic form as $\Phi_r$ in \eqref{eq:renyi-profile-main} since $-k=r/(r-1)$. Here, however, $\Psi_r$ is strictly concave: \eqref{eq:renyi-ratio-main} is its stationarity condition, and case~(ii) occurs when the maximum lies at the endpoint $\tau^\star=\ell_\star$.

In practice, case~(ii) requires no numerical search, since $\tau^\star=\ell_\star$. In case~(i), because the ratio $C_k/C_{k+1}$ is continuous and strictly decreasing, $\tau^\star$ is the unique root of \eqref{eq:renyi-ratio-main} in the bracket $[\ell_\star-\alpha r/(1-r),\,\ell_\star)$, and can be found by any standard one-dimensional root-finding method (costing two integrals against $P$ per evaluation, or one per evaluation if maximizing $\Psi_r$ directly).

\section{Discussion}
\label{sec:conclusion}

The results reduce the unrestricted GVI problem to a scalar problem and integrals against the prior or reference measure.  The explicit formulas provide targets for studying posterior properties and for evaluating numerical or variational approximations. Several questions remain open.

The first numerical question is the effect of approximating these integrals.  In practice the integrals that define the mass function, or $\Phi_r$ and $\Psi_r$ in the R\'enyi case, will be approximated, for instance from prior draws, and the scalar will be computed from the approximation.  For R\'enyi orders above one, Section~S9 shows that the scalar computed from prior draws is consistent when the loss is bounded below and $\Phi_r$ has a unique minimizer, but error bounds for the resulting scalar and density are open.  They are likely to be most delicate for the truncated solutions, whose support depends on the scalar.  A second question is how best to sample from the solutions and compute expectations under them.  Once the scalar is known the density ratio can be evaluated at any point, but efficient methods will depend on the parameter space and on the form of the density.  A third is approximation by tractable variational families, whose error can now be measured against the exact target.

On the statistical side, the explicit targets invite the study of posterior concentration, Bernstein--von Mises limits, coverage, and model-evaluation asymptotics for non-Kullback--Leibler GVI posteriors, in relation to the existing theory for generalized posteriors \citep{miller2021asymptotic,nguyen2026largesample}.

Finally, one extension of the attainment theory remains.  For density power, least squares and Burg, Corollary~S10 in Section~S7 resolves finite feasibility, the required tail control and normalization even on infinite reference measures.  Comparable explicit tests for broader classes of Bregman generators remain to be developed.

\begin{acks}[AI Use Statement]
OpenAI ChatGPT and Codex assisted with manuscript editing, mathematical
formulation, proof and reference checking, and \LaTeX{} formatting.
The authors take responsibility for the manuscript.
\end{acks}

\begin{supplement}
\stitle{Supplementary Material for ``Exact Solutions to Unrestricted Generalized Variational Inference Problems''}
\sdescription{The supplement contains the proofs of all results, the derivations of the rows of Table~\ref{tab:solutions}, and the values of the infimum when finite attainment fails.}
\end{supplement}

\bibliographystyle{ba2}
\bibliography{refs}

\end{document}


\begin{frontmatter}
\title{Supplementary Material for ``Exact Solutions to Unrestricted Generalized Variational Inference Problems''}
\runtitle{Supplementary Material}

\begin{aug}
\author[SA]{\fnms{Jacob}~\snm{Westerhout}\ead[label=se1]{j.westerhout@uq.edu.au}}
\and
\author[SB]{\fnms{Hien Duy}~\snm{Nguyen}\ead[label=se2]{h.nguyen5@latrobe.edu.au}\orcid{0000-0002-9958-432X}}
\address[SA]{School of Mathematics and Physics,
The University of Queensland\printead[presep={,\ }]{se1}}

\address[SB]{Department of Mathematical and Physical Sciences,
La Trobe University\printead[presep={,\ }]{se2}\newline
Institute of Mathematics for Industry, Kyushu University}
\runauthor{J. Westerhout and H. D. Nguyen}
\end{aug}

\end{frontmatter}

Sections, equations, and results in this supplement are numbered with the prefix~S; numbers without the prefix refer to the main paper. Throughout this supplement, density identities hold up to almost-everywhere equality, and the two terms of (1) are evaluated separately, as in Section~2.  A probability measure with infinite divergence has objective $\infty$, probability measures for which the expected loss is undefined are excluded, and a well-defined expected loss of $-\infty$ with finite divergence gives infimum $-\infty$.  No probability measure with finite objective is assumed to exist unless this is stated.

\section{A separable convex integral principle}
\label{subsec:separable-principle}

The $f$-divergence objective uses the density ratio $u=\mathrm{d}Q/\mathrm{d}P$, while the Bregman objective uses a density $q=\mathrm{d}Q/\mathrm{d}\mu$ relative to a dominating measure $\mu$. Both are separable convex integrals subject to a unit-mass constraint.  We first state the principle for an integrand that may depend on $z$, and then specialize to a fixed convex generator.

A function $g:\mathbb{R}\to(-\infty,\infty]$ is proper if it is not identically $\infty$, and lower semicontinuous if all of its sublevel sets $\{t\in\mathbb{R}:g(t)\le c\}$ are closed.  The Fenchel conjugate of $g$ is defined by $g^*(y)=\sup_{t\in\mathbb{R}}\{ty-g(t)\}$ for $y\in\mathbb{R}$.

For a proper convex function $g$, a real number $y$ is a subgradient of $g$ at $t\in\operatorname{dom}g$ if $g(s)\ge g(t)+y(s-t)$ for every $s\in\mathbb{R}$.

Let $\nu$ be a $\sigma$-finite measure on $(\mathcal{Z},\mathfrak{Z})$, and let $I$ be either $[0,\infty)$ or $(0,\infty)$.  Write
\[
\mathcal{U}(\nu)=
\left\{u:\mathcal{Z}\to[0,\infty)\text{ measurable}:
\int_{\mathcal{Z}}u(z)\,\nu(\mathrm{d}z)=1\right\}
\]
for the set of probability densities with respect to $\nu$.  For an integrand $H:\mathcal{Z}\times I\to(-\infty,\infty]$, extended by $\infty$ outside $I$, consider
\begin{equation}
\label{eq:sep-integral-main}
\inf_{u\in\mathcal{U}(\nu)}J_H(u),
\qquad
J_H(u)=\int_{\mathcal{Z}}H(z,u(z))\,\nu(\mathrm{d}z).
\end{equation}
Adding $\lambda(\int_{\mathcal{Z}} u\,\mathrm{d}\nu-1)$ leaves the objective unchanged on $\mathcal{U}(\nu)$ and leads to the pointwise problem of minimizing $H(z,t)+\lambda t$ over $t\in I$.  If its minimizers have unit mass, this solves the constrained problem.

\begin{proposition}[Separable convex integral principle]
\label{thm:sep-main}
Suppose $H$ is jointly measurable and for almost every $z$ the extended function $H(z,\cdot)$ is convex on $[0,\infty)$ and finite on $(0,\infty)$.  A feasible $u^\star$ of finite objective minimizes \eqref{eq:sep-integral-main} if and only if, for some $\lambda\in\mathbb{R}$,
\[
 H(z,u^\star(z))+\lambda u^\star(z)
   =\inf_{t\ge0}\{H(z,t)+\lambda t\}\qquad \nu\text{-almost everywhere}.
\]
Strict convexity of the sections guarantees uniqueness.  The same equivalence holds for $H(z,t)=\ell(z)t+\alpha B(z,t)$ with $B\ge0$ and $\ell$ finite almost everywhere when the two objective terms are evaluated separately.
\end{proposition}
\begin{proof}
A pointwise minimizer gives, for every feasible $u$,
\[
 H(z,u(z))\ge H(z,u^\star(z))+\lambda(u^\star(z)-u(z)).
\]
The right side is integrable.  Integration gives optimality, and equality forces the nonnegative gap to vanish almost everywhere.  With strict convexity the pointwise minimizer is unique.  The same inequality rules out a competitor with finite penalty and expected loss $-\infty$.

For the converse let $u^\star$ be a finite minimizer.  Write $L(z)$ and $R(z)$ for the left and right slopes of $H(z,\cdot)$ at $u^\star(z)$, taking $L(z)=-\infty$ when $u^\star(z)=0$.  Slopes are measurable by one-sided difference quotients, and both are finite wherever $u^\star>0$.  We show that
\[
 \operatorname*{ess\,sup}_\nu L\ \le\ \operatorname*{ess\,inf}_\nu R .
\]
Otherwise choose a real $h$ for which $L>h$ and $R<h$ each hold on sets of positive measure.  These sets are disjoint since $L\le R$.  On a positive-measure subset $A$ of the first set, decreasing $u^\star$ by $d$ reduces the integrand by strictly more than $hd$ for every $0<d\le\delta$.  On a positive-measure subset $B$ of the second, increasing $u^\star$ by $d$ increases the integrand by strictly less than $hd$ for every such $d$.  Here one can choose a common $\delta>0$ by testing $\delta=1/n$ and using monotonicity of convex secant slopes, requiring $u^\star>\delta$ on $A$.

By $\sigma$-finiteness, restrict $A,B$ to sets of finite positive measure.  We may also bound the change in $H$ on the two segments of length $\delta$ there.  Indeed, each segment has finite endpoints and a convex function finite on that segment is bounded there.  The bounds are measurable by testing its endpoints and a countable dense subset.  For the separate-terms objective, also restrict to bounded $|\ell|$, so that each changed term stays integrable.  Now transfer mass $\varepsilon$ from $A$ to $B$, using the constant decrements and increments $\varepsilon/\nu(A)$ and $\varepsilon/\nu(B)$, both at most $\delta$.  This preserves feasibility and strictly lowers the objective, which proves the asserted slope inequality.

The set $\{u^\star>0\}$ has positive measure, so the two essential bounds enclose a finite real number $c$.  Almost everywhere, $L(z)\le c\le R(z)$, which is precisely the supporting-line inequality at $u^\star(z)$, including the right-sided condition at zero.  Taking $\lambda=-c$ gives the result.
\end{proof}

\begin{lemma}[Finiteness from a comparison density]
\label{lem:sep-finite}
Let $H$ satisfy the integrand hypotheses of Proposition~\ref{thm:sep-main}, and suppose a measurable $u_\lambda$ minimizes $H(z,t)+\lambda t$ pointwise and has unit mass.
\begin{enumerate}
\item If $H(z,t)\ge mt$ for some $m\in\mathbb R$ and $J_H(u_0)\in\mathbb R$ for some $u_0\in\mathcal U(\nu)$, then $J_H(u_\lambda)\in\mathbb R$.
\item Suppose $H(z,t)=\ell(z)t+\alpha B(z,t)$, where $\ell$ is real-valued and measurable and $B\ge0$.  If some $u_0\in\mathcal U(\nu)$ satisfies
\[
 \int|\ell|u_0\,\mathrm d\nu<\infty,
 \qquad \int B(\cdot,u_0)\,\mathrm d\nu<\infty,
\]
then $u_\lambda$ has separately finite expected loss and penalty if and only if $\int\ell^-u_\lambda\,\mathrm d\nu<\infty$.
\end{enumerate}
\end{lemma}
\begin{proof}
Pointwise minimality gives
\[
 H(\cdot,u_\lambda)
 \le H(\cdot,u_0)+\lambda(u_0-u_\lambda).
\]
In part~(i), the right side and the lower bound $mu_\lambda$ are integrable.  In part~(ii), add $\ell^-u_\lambda$ to obtain
\[
 0\le\ell^+u_\lambda+\alpha B(\cdot,u_\lambda)
 \le\ell u_0+\alpha B(\cdot,u_0)
       +\lambda(u_0-u_\lambda)+\ell^-u_\lambda.
\]
If the negative-loss integral is finite, the right side is integrable, so both nonnegative terms on the left have finite integrals.  Necessity follows from finite absolute expected loss.
\end{proof}

For a loss in $L^1(P)$, the reference density supplies the comparison law whenever it belongs to the penalty's domain: $u_0=1$ for an $f$-divergence, after replacing its generator by a nonnegative affine representative, and $q_0=p$ for a Bregman divergence.  If the loss is also bounded below, the negative-loss test is automatic.  For the arbitrary measurable losses considered in Table~1, neither property is assumed; the finite-attainment criteria in the table supply exactly the finiteness needed.

\begin{condition}[Separable generator setting]
\label{cond:sep-setup}
Let $\nu$ be a $\sigma$-finite measure on $(\mathcal{Z},\mathfrak{Z})$, and let $I$ be either $[0,\infty)$ or $(0,\infty)$.  Let $\varphi:I\to(-\infty,\infty]$ be proper, lower semicontinuous, strictly convex, and differentiable on $(0,\infty)$.  Let $b:\mathcal{Z}\to\mathbb{R}$ be measurable and let $\alpha>0$.
\end{condition}

Extend $\varphi$ by $\infty$ outside $I$, and denote its Fenchel conjugate by $\varphi^*$.  Taking $H(z,t)=\alpha\varphi(t)+b(z)t$ in \eqref{eq:sep-integral-main} gives
\begin{equation}
\label{eq:sep-primal-main}
\inf_{u\in\mathcal{U}(\nu)}J_{\varphi,b}(u),
\qquad
J_{\varphi,b}(u)=\int_{\mathcal{Z}}\{\alpha\varphi(u(z))+b(z)u(z)\}\,\nu(\mathrm{d}z).
\end{equation}
The pointwise problem is now
\[
\inf_{t\in I}\{\alpha\varphi(t)+(b(z)+\lambda)t\}.
\]
Its solution is given by a single-valued version of the map $S_\varphi$ from Section~3.  Write $\varphi'_+(0)=\lim_{t\downarrow0}\varphi'(t)$ and $\varphi'_-(\infty)=\lim_{t\uparrow\infty}\varphi'(t)$, and define
\begin{equation}
\label{eq:scalar-resolvent-inverse-main}
T_\varphi(y)=
\begin{cases}
0, & 0\in I\text{ and }y\le\varphi'_+(0),\\[0.5em]
(\varphi')^{-1}(y), & \varphi'_+(0)<y<\varphi'_-(\infty).
\end{cases}
\end{equation}
The map is left undefined outside these cases.  Whenever it is defined, $T_\varphi(y)$ is the unique minimizer of $\varphi(t)-yt$ over $t\in I$, so that $S_\varphi(y)=\{T_\varphi(y)\}$; otherwise $S_\varphi(y)$ is empty.  The zero branch is the source of the positive parts in Table~1, and the restriction $y<\varphi'_-(\infty)$ is the source of the positivity requirements in the rows with reciprocal densities.  The pointwise candidate is therefore
\begin{equation}
\label{eq:sep-u-lambda-main}
u_\lambda(z)=T_\varphi\!\left(-\frac{b(z)+\lambda}{\alpha}\right).
\end{equation}
Whenever the map is defined, it gives the unique pointwise minimizer.  By Proposition~\ref{thm:sep-main}, any measurable, normalized candidate with finite objective is the unique optimizer.  Lemma~\ref{lem:sep-finite} supplies finiteness whenever its hypotheses hold.

\begin{corollary}[Separable dual identity]
\label{cor:sep-dual-main}
Let $\nu$ be $\sigma$-finite, $b$ measurable and real-valued, and $\varphi$ convex on $[0,\infty)$, finite on $(0,\infty)$ and possibly infinite at zero.  If \eqref{eq:sep-primal-main} has a finite minimizer $u^\star$, then for any multiplier $\lambda^\star$ certifying it in Proposition~\ref{thm:sep-main},
\begin{equation}
\label{eq:sep-dual-main}
\inf_{u\in\mathcal{U}(\nu)}J_{\varphi,b}(u)
=
\sup_{\lambda\in\mathbb{R}}
\left\{-\lambda-\alpha\int_{\mathcal{Z}}\varphi^*\!\left(-\frac{b(z)+\lambda}{\alpha}\right)\,\nu(\mathrm{d}z)\right\}.
\end{equation}
The supremum is attained at $\lambda^\star$.
\end{corollary}
\begin{proof}
By the definition of the conjugate,
\[
\inf_{t\in I}\{\alpha\varphi(t)+(b(z)+\lambda)t\}
=-\alpha\varphi^*\!\left(-\frac{b(z)+\lambda}{\alpha}\right).
\]
Using $t=u^\star(z)$ in the supremum defining $\varphi^*$ bounds $\varphi^*(-(b+\lambda)/\alpha)$ below by the integrable function $-[H(\cdot,u^\star)+\lambda u^\star]/\alpha$.  Hence the integral of the conjugate term is well defined.  For every feasible $u$ and every $\lambda\in\mathbb{R}$,
\[
J_{\varphi,b}(u)
=\int_{\mathcal{Z}}\{\alpha\varphi(u)+(b+\lambda)u\}\,\mathrm{d}\nu-\lambda
\ge-\lambda-\alpha\int_{\mathcal{Z}}\varphi^*\!\left(-\frac{b+\lambda}{\alpha}\right)\,\mathrm{d}\nu.
\]
Equality holds at $u=u^\star$ and $\lambda=\lambda^\star$ by pointwise minimality.
\end{proof}

\section{Proof of Theorem~1}
\label{app:proof-f-main}

Every probability measure with finite divergence is absolutely continuous with respect to $P$, so it suffices to consider density ratios $u\in\mathcal{U}(P)$. For $H(z,t)=\ell(z)t+\alpha f(t)$, membership in (6) is exactly pointwise minimization of $H(z,t)+\lambda t$. The separate-terms version of Proposition~\ref{thm:sep-main} therefore gives the equivalence under (7). Its converse proof uses mass transfers on sets of bounded loss and does not require $\ell\in L^1(P)$.

For the assertion about $Q_0$, recenter $f$ with a supporting line at $1$ so that $f\ge0$. This changes only the multiplier. Let $u_0=\mathrm{d}Q_0/\mathrm{d}P$. Pointwise minimality gives
\[
 0\le\alpha f(u)+\ell^+u
 \le \ell u_0+\alpha f(u_0)+\lambda(u_0-u)+\ell^-u.
\]
If $u$ has unit mass and $\int\ell^-u\,\mathrm dP<\infty$, the right side is integrable. Thus its positive expected loss and divergence are finite as well. Conversely, (7) implies $\int\ell^-u\,\mathrm dP<\infty$.

For any other finite optimizer, the nonnegative pointwise gap associated with a certifying multiplier integrates to zero. Strict convexity gives uniqueness of the density. If $f$ is differentiable at positive densities, every certifying multiplier equals $-\ell-\alpha f'(u)$ on the positive-mass set where a fixed finite optimizer has $u>0$. This also gives uniqueness of the multiplier. Without differentiability the multiplier need not be unique: for $f(t)=(t-1)^2+|t-1|$ and $\ell=0$, $u=1$ is the unique optimizer and every $\lambda\in[-\alpha,\alpha]$ certifies it.

Theorem~1 characterizes finite minimizers without asserting that one exists. The exact existence conditions for the generators in Table~1 are derived in Sections~\ref{app:f-examples} and~\ref{app:reverse-mixed}. For a general generator, the following condition is sufficient for existence when the loss is integrable under the prior.

\begin{proposition}[A sufficient condition for existence]
\label{prop:f-existence}
Suppose that $\ell\in L^1(P)$, that $f$ is lower semicontinuous at zero, and that $f(t)/t\to\infty$ as $t\to\infty$. A function $u$ and a scalar $\lambda$ satisfying (6) and (7) exist if, for some $0<\beta<\alpha$,
\begin{equation}
 \int f^*(\ell^-/\beta)\,\mathrm dP<\infty,\qquad
 f^*(y)=\sup_{t\ge0}\{yt-f(t)\}.
 \label{eq:f-tail-sufficient}
\end{equation}
In particular, this condition holds when $\ell\in L^1(P)$ is bounded below.
\end{proposition}
\begin{proof}
The extension of $f$ to $\mathbb R$ is lower semicontinuous. Superlinearity makes $f^*(y)=\sup_{t\ge0}\{yt-f(t)\}$ finite for every $y$, and its maximizers form a nonempty compact interval $[L(y),U(y)]$. Its endpoints are the left and right derivatives of $f^*$, respectively. Indeed, comparison of maximizers at $y$ and $y+h$ bounds the right difference quotient between $U(y)$ and $U(y+h)$ for $h>0$. As $h\downarrow0$, these maximizers remain bounded by superlinearity, and every limit point maximizes at $y$ by lower semicontinuity. The left derivative follows in the same way. The endpoints are consequently measurable and nondecreasing.

Consider
\[
 F(c)=\int f^*(c-\ell/\alpha)\,\mathrm dP-c.
\]
Writing $h=\beta/\alpha$, monotonicity and convexity of $f^*$ give
\[
 c-\ell/\alpha\le f^*(c-\ell/\alpha)
 \le h f^*(\ell^-/\beta)+(1-h)f^*\bigl(c/(1-h)\bigr).
\]
The lower bound tests the conjugate at $1$. The upper bound follows from $c-\ell/\alpha\le c+\ell^-/\alpha$. Since $\ell\in L^1(P)$ and \eqref{eq:f-tail-sufficient} holds, $F$ is finite and convex. Testing the conjugate at $t=1/2$ and $t=2$ gives
\[
 F(c)\ge(t-1)c-t\int\ell\,\mathrm dP/\alpha-f(t),
\]
which proves coercivity in both directions. At a minimizing $c$, its one-sided derivatives give
\[
 \int L(c-\ell/\alpha)\,\mathrm dP\le1
 \le\int U(c-\ell/\alpha)\,\mathrm dP.
\]
To justify differentiating the integral, put $y=c-\ell/\alpha$. Difference quotients of $f^*$ with increments of absolute value less than one are nonnegative and bounded above by $f^*(y+1)-f^*(y)$, an integrable function by the preceding bounds. A constant convex combination $u$ of the two endpoint selections therefore has unit mass and belongs pointwise to the maximizing interval. Set $\lambda=-\alpha c$.

Testing pointwise optimality against $1$ and using Fenchel--Young,
$\ell^-u\le\beta f(u)+\beta f^*(\ell^-/\beta)$, gives
\[
 (\alpha-\beta)f(u)+\ell^+u
 \le\ell+\lambda(1-u)+\beta f^*(\ell^-/\beta).
\]
The right side is integrable. The negative part of $f(u)$ is integrable because $f(u)\ge s(u-1)$ for $s\in\partial f(1)$ and $u$ has unit mass. It follows that $f(u)$ and $\ell^+u$ are integrable. Fenchel--Young then gives $\int\ell^-u\,\mathrm dP<\infty$, so $u$ and $\lambda$ satisfy (6) and (7). When $\ell\in L^1(P)$ is bounded below, \eqref{eq:f-tail-sufficient} follows because $f^*$ is finite and nondecreasing.
\end{proof}

Lower semicontinuity at zero matters for this existence guarantee. On two equiprobable points take losses $(0,6\alpha)$ and $f(t)=(t-1)^2$ for $t>0$, with $f(0)=2$. This generator is strictly convex and superlinear. For density ratios $(2-v,v)$ with $0<v<2$, the objective is $\alpha(1+v+v^2)$. Each endpoint incurs an extra $\alpha/2$ relative to the continuous extension. The infimum $\alpha$ is approached as $v\downarrow0$ and is not attained.

\section[Measurable-loss criteria for the f-divergence rows]{Measurable-loss criteria for the $f$-divergence rows}
\label{app:f-examples}

The following corollary of Theorem~1 states the finite-attainment conditions in the $f$-divergence block of Table~1, and makes the normalizability requirement in its reciprocal rows explicit. It uses the notation $X$, $M_c$ and $L_+$ from (3).

\begin{corollary}[Measurable-loss criteria for the $f$-divergence rows]
\label{cor:f-measurable}
Let $\ell:\mathcal Z\to\mathbb R$ be measurable. The following conditions are necessary and sufficient for finite attainment in the corresponding rows of Table~1:
\begin{enumerate}
\item[\textup{(i)}] Forward Kullback--Leibler: $M_\alpha<\infty$.
\item[\textup{(ii)}] Cressie--Read of order $\gamma>1$: $X\in L^{\gamma/(\gamma-1)}(P)$. For Pearson $\chi^2$, the condition is $X\in L^2(P)$.
\item[\textup{(iii)}] Squared Hellinger: the displayed density is normalizable.
\item[\textup{(iv)}] Reverse Kullback--Leibler: the displayed density is normalizable and $L_+<\infty$.
\item[\textup{(v)}] Mixed Kullback--Leibler with $a,b>0$: $M_{a\alpha}<\infty$ and $L_+<\infty$.
\end{enumerate}
Here normalizability includes the requirement $\alpha+\ell+\lambda>0$ almost surely in the reciprocal rows. Whenever the stated condition holds, the displayed density and its normalizing scalar are unique.

More explicitly, put $m=\operatorname*{ess\,inf}_P\ell$. For squared Hellinger and reverse Kullback--Leibler, respectively, put $j=2$ and $j=1$. The corresponding density is normalizable exactly when $m> -\infty$ and
\begin{equation}
 B_j:=\int\!\left(\frac{\alpha}{\ell-m}\right)^j\,\mathrm dP\ge1,
 \qquad\frac1{0}=\infty.
 \label{eq:f-endpoint-test}
\end{equation}
Equality gives the boundary value $\lambda=-\alpha-m$. A strict inequality gives a unique root in $(-\alpha-m,\infty)$.
\end{corollary}

No moment condition on $\ell^+$ is required for forward KL, Cressie--Read, Pearson or Hellinger. For reverse and mixed KL, the logarithmic moment condition $L_+<\infty$ is necessary for any finite feasible objective. Boundedness below is not assumed, but it is a necessary consequence of finite attainment for Hellinger and reverse KL. For an example of equality in \eqref{eq:f-endpoint-test}, let the loss have density $(j+1)x^j$ on $(0,1)$. Then $B_j=(j+1)\alpha^j$, so attainment holds precisely for $\alpha\ge(j+1)^{-1/j}$, including equality.

We prove Corollary~\ref{cor:f-measurable} while deriving the corresponding rows of Table~1. The reverse and mixed rows are completed in Section~\ref{app:reverse-mixed}. Throughout, $y=-(\ell+\lambda)/\alpha$. The generators for the four forward rows are
\[
\begin{aligned}
f_{\mathrm{KL}}(t)&=t\log t-t+1,&
f_\gamma(t)&=\frac{t^\gamma-\gamma t+\gamma-1}{\gamma(\gamma-1)},\quad\gamma>1,\\
f_{\chi^2}(t)&=(t-1)^2,&
f_{\mathrm H}(t)&=(\sqrt t-1)^2,
\end{aligned}
\]
for $t\ge0$, with $0\log0=0$. In each of these four cases a finite comparison density always exists: choose $n$ such that $P(|\ell|\le n)>0$ and condition $P$ on that set. Its loss is bounded on its support, and its divergence is finite because the generator is finite at zero and at every positive number. Thus the assertion about $Q_0$ in Theorem~1 applies without any assumption on $\int|\ell|\,\mathrm dP$.

\paragraph{Forward Kullback--Leibler.}
Here $f'(t)=\log t$, so $T_f(y)=e^y$ for every real $y$ and
$u_\lambda=e^{-(\ell+\lambda)/\alpha}$. Normalization is possible exactly when $Z_\alpha=\int e^{-\ell/\alpha}\,\mathrm dP<\infty$, and then
\[
 \lambda=\alpha\log Z_\alpha,\qquad
 u_\lambda=e^{-\ell/\alpha}/Z_\alpha.
\]
Since $\ell$ is real almost surely, $Z_\alpha>0$. The condition $M_\alpha<\infty$ is equivalent to finite Gibbs mass together with $\int Xu_\lambda\,\mathrm dP<\infty$. Indeed, $e^{X/\alpha}$ is bounded on $\{X\le1\}$ and is bounded by $Xe^{X/\alpha}$ on $\{X>1\}$. The positive expected loss is automatically finite because $t e^{-t/\alpha}$ is bounded for $t\ge0$. Finally, on the Gibbs density,
$\log u_\lambda=-\ell/\alpha-\log Z_\alpha$, so finite expected absolute loss also gives finite relative entropy. This proves Corollary~\ref{cor:f-measurable}(i).

\paragraph{Cressie--Read and Pearson.}
For $\gamma>1$,
\[
 f'_\gamma(t)=\frac{t^{\gamma-1}-1}{\gamma-1},\qquad
 f'_{\gamma,+}(0)=-\frac1{\gamma-1},\qquad
 f'_{\gamma,-}(\infty)=\infty.
\]
The inverse map is
$T_{f_\gamma}(y)=(1+(\gamma-1)y)_+^{1/(\gamma-1)}$.
Put $k=1/(\gamma-1)$. For each fixed multiplier, the candidate grows as a positive constant times $X^k$ as $\ell\to-\infty$. Thus
\[
 \int Xu_\lambda\,\mathrm dP<\infty
 \quad\Longleftrightarrow\quad
 \int X^{k+1}\,\mathrm dP<\infty
\]
whenever it normalizes. Under this moment condition, the candidates on any compact multiplier interval are dominated by $C(1+X^k)$ for a suitable constant $C$. Their mass is therefore continuous. It decreases strictly wherever positive, tends to zero as $\lambda\to\infty$ by domination by a fixed candidate, and tends to infinity as $\lambda\to-\infty$ by monotone convergence. Hence there is a unique normalizer. The upper loss tail is truncated at $\alpha/(\gamma-1)-\lambda$, and the divergence is finite because $u_\lambda^\gamma=O(1+X^{k+1})$.

For Pearson, $f'(t)=2(t-1)$ and $T_f(y)=(1+y/2)_+$. The same argument applies with $k=1$, yielding exactly $X\in L^2(P)$. The truncation threshold is $2\alpha-\lambda$. This proves Corollary~\ref{cor:f-measurable}(ii).

\paragraph{Squared Hellinger and the reciprocal mass test.}
Here $f'(t)=1-t^{-1/2}$, with derivative limits $-\infty$ and $1$. Consequently $T_f(y)=(1-y)^{-2}$ is defined exactly for $y<1$, and the candidate is
$u_\lambda=\{\alpha/(\alpha+\ell+\lambda)\}^2$.
If $\alpha+\ell+\lambda\le0$, the pointwise Lagrangian
$(\alpha+\ell+\lambda)t-2\alpha\sqrt t+\alpha$ has no minimizer. Thus an admissible multiplier forces $m=\operatorname*{ess\,inf}_P\ell> -\infty$.

For either exponent $j=2$ or $j=1$, the mass
\[
 G_j(\lambda)=\int\left(\frac{\alpha}{\alpha+\ell+\lambda}\right)^j\,\mathrm dP
\]
is finite, continuous and strictly decreasing on $(-\alpha-m,\infty)$, and tends to zero at infinity. These facts follow by domination by a constant on each compact multiplier interval. Monotone convergence at the left endpoint gives $B_j$ in \eqref{eq:f-endpoint-test}. If $B_j=1$, then $P(\ell=m)=0$, and the endpoint candidate itself is defined almost surely and has unit mass. This proves the claimed exact normalizability test, including the boundary case.

For a normalized Hellinger candidate, the negative expected loss is finite because $m> -\infty$. The positive expected loss is finite because $\ell u_\lambda=O(\ell^{-1})$ on the upper tail, and on bounded-loss sets it is bounded by a constant times $u_\lambda$. The divergence is bounded by $2$ for every probability density. This proves Corollary~\ref{cor:f-measurable}(iii), without any upper-tail assumption. On a finite space with positive prior masses, $B_j=\infty$ because the minimum-loss set has positive mass, so both reciprocal formulas normalize.

\section[Derivations for the reverse and mixed Kullback--Leibler rows]{Derivations for the reverse and mixed Kullback--Leibler rows}
\label{app:reverse-mixed}

For $u=\mathrm dQ/\mathrm dP>0$,
$\mathrm dP/\mathrm dQ=1/u$ and $\mathrm dQ=u\,\mathrm dP$, so
\[
 \mathsf D_f(P\Vert Q)
 =\int f(1/u)u\,\mathrm dP
 =\int f^\diamond(u)\,\mathrm dP,
 \qquad f^\diamond(t)=tf(1/t).
\]
The convexity of the perspective \citep[Sec.~3.2.6]{boyd2004convex} gives convexity of $f^\diamond$, and setting $f^\diamond(0)=\infty$ restricts the problem to positive density ratios. Theorem~1 applies to $f^\diamond$, including when this extension at zero is not lower semicontinuous, and gives (6) with $f^\diamond$ in place of $f$ at every finite optimum. In the strictly convex differentiable case this reads
\[
  u^\star(z)=\big((f^\diamond)'\big)^{-1}\!\left(-\frac{\ell(z)+\lambda^\star}{\alpha}\right),
  \qquad \int u^\star\,\mathrm dP=1.
\]
The finite-objective conditions are verified below rather than inferred from integrability under the prior.

\paragraph{The necessary logarithmic upper moment.}
For any density with finite reverse KL and $\int\ell^+u\,\mathrm dP<\infty$, positivity of $u$ and $\log t\le t-1$ give
\[
 \log(1+\ell^+)
 =\log\{(1+\ell^+)u\}-\log u
 \le(1+\ell^+)u-1-\log u.
\]
The right side is integrable, since the negative part of $-\log u$ is bounded by $u$. Hence $L_+<\infty$ is necessary for any finite feasible objective with reverse KL or mixed KL.

\paragraph{Reverse Kullback--Leibler.}
Here $f^\diamond(t)=-\log t+t-1$ and
$(f^\diamond)'(t)=1-1/t$. Its inverse is $1/(1-y)$ for $y<1$, giving
\[
 u_\lambda=\frac{\alpha}{\alpha+\ell+\lambda},
 \qquad\alpha+\ell+\lambda>0\quad P\text{-almost surely}.
\]
The reciprocal mass test was proved above. If this candidate normalizes, its negative expected loss is finite because $m> -\infty$, and its positive expected loss is finite because $\ell u_\lambda$ is bounded on the upper tail. Its reverse divergence has integrand
$\log\{(\alpha+\ell+\lambda)/\alpha\}$.
The positive part of this logarithm is integrable exactly when $L_+<\infty$; its negative part is bounded by $u_\lambda$. These observations prove Corollary~\ref{cor:f-measurable}(iv).

\paragraph{Mixed Kullback--Leibler.}
For $a,b>0$, use
\[
 f_{a,b}(t)=a(t\log t-t+1)+b(-\log t+t-1),
 \qquad f'_{a,b}(t)=a\log t+b(1-1/t).
\]
The derivative is a bijection from $(0,\infty)$ onto $\mathbb R$. In the equation $f'_{a,b}(u)=y$, set $v=b/(au)$. Then
\[
 \log v+v=\log(b/a)+\frac{b-y}{a},\qquad
 u=\frac{b/a}{\mathrm W\bigl((b/a)e^{(b-y)/a}\bigr)}.
\]
This gives the table entry. For numerical evaluation, solving the strictly increasing equation for $\log v+v$ avoids forming the exponential inside the Lambert function.

Write $T=T_{f_{a,b}}$. The defining equation gives
\[
 \frac{T(y)}{e^{y/a}}\longrightarrow e^{-b/a}
       \quad(y\to\infty),\qquad
 (-y)T(y)\longrightarrow b\quad(y\to-\infty).
\]
The first limit follows from
$u=\exp\{(y-b+b/u)/a\}$; for the second, multiply
$y=a\log u+b-b/u$ by $u$ and let $u\downarrow0$.
Consequently the candidate mass is finite for one, equivalently every, multiplier exactly when $Z_{a\alpha}<\infty$. Under this condition, the candidates on compact multiplier intervals are dominated by $C(1+e^{X/(a\alpha)})$. The same continuity and endpoint argument used for the power densities yields a unique normalizer.

At this normalizer, the negative expected loss is finite exactly when $M_{a\alpha}<\infty$. The positive expected loss is always finite, since $\ell u_\lambda\to b\alpha$ as $\ell\to\infty$. The forward-KL integrand is integrable exactly when the lower-tail moment $M_{a\alpha}$ is finite, while the reverse-KL integrand is integrable exactly when $L_+<\infty$: its upper-loss-tail growth is $\log(1+\ell^+)$, and its negative part is bounded by $u_\lambda$. This proves Corollary~\ref{cor:f-measurable}(v). Strict convexity and differentiability of all six generators give the stated uniqueness conclusions by Theorem~1.

\section[Values of the infimum when finite attainment fails in the f-divergence rows]{Values of the infimum when finite attainment fails in the $f$-divergence rows}
\label{app:f-value-regimes}

When the conditions of Corollary~\ref{cor:f-measurable} fail, the infimum in (1) is either $-\infty$ or finite and unattained. The next proposition identifies which case occurs for each $f$-divergence row of Table~1.

\begin{proposition}[Values and failure of finite attainment]
\label{prop:f-value-regimes}
Retain the measurable real-valued loss and the notation of Corollary~\ref{cor:f-measurable}, and put
\[
 Z_c=\int e^{-\ell/c}\,\mathrm dP\in(0,\infty]\qquad(c>0).
\]
For reverse and mixed KL, $L_+=\infty$ rules out every finite feasible objective. The assertions below about these two penalties assume $L_+<\infty$.
\begin{enumerate}
\item[\textup{(i)}] For forward KL the infimum is $-\alpha\log Z_\alpha$, with value $-\infty$ if $Z_\alpha=\infty$. If $Z_\alpha<\infty$ but $M_\alpha=\infty$, the infimum is finite and unattained, although the Gibbs density normalizes.
\item[\textup{(ii)}] For Cressie--Read and Pearson, failure of the moment condition in Corollary~\ref{cor:f-measurable}(ii) gives infimum $-\infty$.
\item[\textup{(iii)}] For Hellinger and reverse KL, $m=-\infty$ gives infimum $-\infty$. If $m> -\infty$ and $B_j<1$, the infimum is finite and unattained. Writing $w=\ell-m$, these unattained values are
\begin{equation}
 \begin{aligned}
 \inf J_{\mathrm H}&=m+2\alpha-\alpha^2\int w^{-1}\,\mathrm dP,
       &&B_2<1,\\
 \inf J_{\mathrm{rev}}&=m+\alpha+\alpha\int\log(w/\alpha)\,\mathrm dP,
       &&B_1<1.
 \end{aligned}
 \label{eq:f-boundary-values}
\end{equation}
\item[\textup{(iv)}] For mixed KL, $Z_{a\alpha}=\infty$ gives infimum $-\infty$. If $Z_{a\alpha}<\infty$, there is a unique normalized candidate $u_\lambda$ from the table, and the infimum is the finite number
\begin{equation}
 -\lambda-\alpha b\int\log u_\lambda\,\mathrm dP.
 \label{eq:mixed-value}
\end{equation}
It is attained if and only if $M_{a\alpha}<\infty$.
\end{enumerate}
\end{proposition}

The assertions that the infimum equals $-\infty$ use the separate-terms convention from the beginning of this supplement; in some cases a finite-divergence density already has expected loss $-\infty$. We now prove the proposition. Write $E$ for integration under $P$. The logarithmic-moment argument above already proves the assertion about $L_+=\infty$.

\paragraph{Part (i): forward KL.}
If $Z_\alpha<\infty$, write $u_G=e^{-\ell/\alpha}/Z_\alpha$. For every density with separately finite objective terms,
\[
 E(\ell u)+\alpha\mathsf{KL}(Q\Vert P)
 =-\alpha\log Z_\alpha+\alpha\mathsf{KL}(Q\Vert Q_G).
\]
A density with finite divergence cannot have expected loss $-\infty$: bounded exponential tilts give
\[
 E[u\min\{X,N\}]
 \le\alpha\mathsf{KL}(Q\Vert P)+\alpha\log E e^{X/\alpha},
\]
and the right side is finite because $Ee^{X/\alpha}\le1+Z_\alpha$. Monotone convergence gives $EXu<\infty$. Thus the displayed lower bound holds for every admissible objective, including the separately evaluated extended cases.

Let $A_n=\{|\ell|\le n\}$ and $Z_{\alpha,n}=E[e^{-\ell/\alpha}1_{A_n}]$. For all sufficiently large $n$, the truncated Gibbs densities
$u_n=e^{-\ell/\alpha}1_{A_n}/Z_{\alpha,n}$ have finite objective exactly $-\alpha\log Z_{\alpha,n}$. Since $Z_{\alpha,n}\uparrow Z_\alpha$, they prove the claimed infimum both when $Z_\alpha<\infty$ and when $Z_\alpha=\infty$. Equality in the lower bound forces $u=u_G$, so the finite-but-unattained case follows from Corollary~\ref{cor:f-measurable}(i).

\paragraph{Part (ii): Cressie--Read and Pearson.}
Put $h=\gamma/(\gamma-1)$ and suppose $EX^h=\infty$. Choose a bounded density $b_0$ supported on a positive-probability set of bounded losses. Set $Y_n=X1_{\{X\le n\}}$, and, for sufficiently large $n$, put
\[
 v_n=\frac{Y_n^{h-1}}{(EY_n^h)^{1/\gamma}},
 \qquad
 u_n=v_n+(1-Ev_n)b_0.
\]
Then $\|v_n\|_{L^\gamma(P)}=1$, so $Ev_n\le1$ and $u_n$ has unit mass. Moreover,
\[
 \|u_n\|_{L^\gamma(P)}\le1+\|b_0\|_{L^\gamma(P)},
 \qquad E(Xv_n)=(EY_n^h)^{1/h}\longrightarrow\infty.
\]
Every $u_n$ has finite expected loss and finite divergence; their positive expected losses and divergences stay bounded, whereas their negative expected losses diverge. This proves infimum $-\infty$. The Pearson case is the same argument with $\gamma=2$.

\paragraph{Part (iii): Hellinger and reverse KL.}
If $m=-\infty$, choose conditional-prior densities $v_n$ on positive-probability sets of bounded losses below $-n$. Their Hellinger divergences are at most $2$, giving infimum $-\infty$. For reverse KL assume $L_+<\infty$ and set
\[
 u_0=\frac{c_0}{1+\ell^+},\qquad
 c_0=\left(E\frac1{1+\ell^+}\right)^{-1}.
\]
This positive bounded density has finite forward and reverse KL, finite positive expected loss, and $EXu_0=c_0EX$. If $EX=\infty$, its expected loss is already $-\infty$. If $EX<\infty$, mix it equally with $v_n$. The resulting expected losses tend to $-\infty$, while
\[
 \mathsf{KL}\bigl(P\Vert (Q_0+V_n)/2\bigr)
 \le\mathsf{KL}(P\Vert Q_0)+\log2.
\]
Each mixture has finite expected loss and finite divergence, proving the assertion.

Now suppose $m> -\infty$ and $B_j<1$. Put $w=\ell-m$ and
$v=(\alpha/w)^j$, so $P(w=0)=0$ and $\theta=Ev=B_j<1$. For Hellinger, completion of the square gives, for every probability density,
\[
 J_{\mathrm H}(u)
 =m+2\alpha+E(wu-2\alpha\sqrt u)
 \ge m+2\alpha-\alpha^2E(w^{-1}).
\]
The last expectation is finite by $B_2<1$ and Cauchy--Schwarz. Equality would require $u=v$, which has mass less than one. For reverse KL, pointwise minimization of $wu-\alpha\log u$ gives
\[
 J_{\mathrm{rev}}(u)
 \ge m+\alpha+\alpha E\log(w/\alpha).
\]
This logarithm is integrable: its positive part is controlled by $L_+$ and its negative part by $E(w^{-1})<\infty$. Again equality would require $u=v$.

For either penalty let $E_\varepsilon=\{0<w<\varepsilon\}$ and
\[
 u_\varepsilon=v+\frac{1-\theta}{P(E_\varepsilon)}1_{E_\varepsilon}.
\]
These sets have positive probability, because $m$ is the essential infimum and $P(w=0)=0$, and their probability decreases to zero as $\varepsilon\downarrow0$, because they decrease to the empty set. The added mass contributes at most $(1-\theta)\varepsilon$ to the expected excess loss $E(wu_\varepsilon)$. In the Hellinger case,
\[
 0\le E\sqrt{u_\varepsilon}-E\sqrt v
 \le\sqrt{(1-\theta)P(E_\varepsilon)}\longrightarrow0.
\]
For reverse KL,
\[
 0\le E\log(u_\varepsilon/v)
 \le P(E_\varepsilon)
 \log\left(1+\frac{(1-\theta)\varepsilon}{\alpha P(E_\varepsilon)}\right)
 \longrightarrow0.
\]
These are finite competitors and approach the corresponding lower bounds. This proves \eqref{eq:f-boundary-values} and non-attainment.

\paragraph{Part (iv): mixed KL.}
Assume $L_+<\infty$ and first suppose $Z_{a\alpha}=\infty$. Use the density $u_0$ just constructed. If $EX=\infty$, it already has value $-\infty$ and finite mixed divergence. Otherwise it has a finite objective. Let $v_n$ be the truncated Gibbs densities for forward KL with weight $a\alpha$. Their forward-KL objectives tend to $-\infty$ by part~\textup{(i)}. Convexity of forward KL and the preceding reverse-KL bound show that the mixed objectives at $(u_0+v_n)/2$ tend to $-\infty$ as well.

Next suppose $Z_{a\alpha}<\infty$. The preceding derivation gives a unique normalized candidate $u=u_\lambda$. Its logarithm is integrable because $L_+<\infty$ and $(\log u)_+\le u$. For the nonnegative mixed generator used above,
\[
 f_{a,b}(u)-u f'_{a,b}(u)=a(1-u)-b\log u.
\]
Since $\ell+\lambda=-\alpha f'_{a,b}(u)$, its joint pointwise integrand satisfies
\[
 H(\cdot,u)=\ell u+\alpha f_{a,b}(u)
 =-\lambda u+\alpha a(1-u)-\alpha b\log u.
\]
The right side is integrable even if the two objective terms on the left are not separately integrable. Pointwise optimality implies
\[
 H(\cdot,v)\ge H(\cdot,u)+\lambda(u-v)
\]
for every density $v$. The right side is integrable, so integration gives the lower bound \eqref{eq:mixed-value} for every admissible separately evaluated objective and rules out a finite-divergence competitor with expected loss $-\infty$.

This lower bound is approachable by densities with separately finite terms. Here $EX<\infty$, so $u_0=c_0/(1+\ell^+)$ has a finite mixed objective. Put
\[
 \widetilde u_n=u1_{\{|\ell|\le n\}}+u_01_{\{|\ell|>n\}}.
\]
Its mass tends to one. Fix a set $A$ of positive probability on which the loss is bounded and $u$ is bounded above and away from zero, and correct the mass by adding
$\{1-E\widetilde u_n\}1_A/P(A)$. For all sufficiently large $n$, this gives a positive probability density $u_n$ with finite separate terms. The integrals of $H(\cdot,\widetilde u_n)$ converge to that of $H(\cdot,u)$ by integrability of $H(\cdot,u)$ and $H(\cdot,u_0)$. The correction changes that integral by a quantity tending to zero, by continuity on the bounded-loss set $A$. Thus $J(u_n)$ converges to \eqref{eq:mixed-value}. Finally, strict pointwise convexity forces any finite optimizer to equal $u_\lambda$. Corollary~\ref{cor:f-measurable}(v) gives the stated attainment condition.

\paragraph{Normalization need not give a finite solution.}
Two examples show why normalizability alone cannot replace the moment conditions in Table~1. For forward KL, take $P(n)\propto e^{-n}/n^2$ on the positive integers and $\ell(n)=-\alpha n$. The normalized Gibbs law has masses proportional to $n^{-2}$, infinite relative entropy and expected loss $-\infty$. Here $Z_\alpha<\infty$ but $M_\alpha=\infty$, so the infimum $-\alpha\log Z_\alpha$ is finite and unattained. There is a different failure for Pearson: take $P(n)\propto n^{-3}$ and $\ell(n)=-n$. The positive-part candidate has a unique normalizer because $E_Pn<\infty$, but its expected absolute loss and divergence are both infinite because $E_Pn^2=\infty$. In that example the infimum is $-\infty$.

\section{Proof and interpretation of Theorem~2}
\label{app:proof-breg-main}

The Bregman objective uses the integrand
\begin{equation}
\label{eq:breg-integrand-main}
 H(z,t)=\ell(z)t+\alpha B(z,t),\qquad
 B(z,t)=\psi(t)-\psi(p(z))-s(z)(t-p(z)).
\end{equation}
The subgradient inequality makes $B$ nonnegative, and $B$ vanishes at $t=p(z)$.  Where $\psi$ has a kink at $p(z)$, the choice of $s(z)$ is part of the penalty.  Where $p=0$, the usual choice is the right derivative of $\psi$ at zero when it is finite, and no supporting slope at zero is needed if $p>0$ almost everywhere.  The full discrepancy $B$ remains meaningful even when the prior-dependent terms have infinite integrals.  For instance, with $\psi(t)=-\log t$, the function $\psi(p)-\psi'(p)p=1-\log p$ need not be integrable on an unbounded Lebesgue space.

\begin{proof}[Proof of Theorem~2]
Every probability measure with finite divergence is absolutely continuous with respect to $\mu$, so it suffices to consider densities $q\in\mathcal{U}(\mu)$.  For fixed $z$, minimizing $H(z,t)+\lambda t$ is equivalent to minimizing $\psi(t)-[s(z)-(\ell(z)+\lambda)/\alpha]t$.  Proposition~\ref{thm:sep-main}, with the two objective terms evaluated separately, therefore gives (9) for every finite minimizer.  Conversely, a function $q$ satisfying (9) and (10) has unit mass and finite objective and minimizes the scalar Lagrangian pointwise, so the same proposition gives optimality.  Strict convexity gives uniqueness of the density.  If $\psi$ is differentiable on $(0,\infty)$, every certifying multiplier equals $-\ell+\alpha\{s-\psi'(q)\}$ on the positive-measure set where a fixed finite optimizer has $q>0$, and every certifying multiplier certifies every finite optimizer, because the nonnegative pointwise gap integrates to zero.  This gives uniqueness of the multiplier.

For the assertion about $Q_0$, let $q_0=\mathrm{d}Q_0/\mathrm{d}\mu$, which has finite absolute expected loss and finite discrepancy.  Compare a function $q$ of unit mass satisfying (9) with $q_0$ in the pointwise optimality inequality to obtain
\[
 0\le \ell^+q+\alpha B(\cdot,q)
 \le \ell q_0+\alpha B(\cdot,q_0)
       +\lambda(q_0-q)+\ell^-q.
\]
If $\int\ell^-q\,\mathrm{d}\mu<\infty$, the right side is integrable.  Thus the positive expected loss and the discrepancy are both finite.  Necessity of the negative-loss condition follows directly from (10).  This argument uses $q_0$, not any integrability of the prior loss.
\end{proof}

\begin{remark}[Mirror-descent interpretation]
\label{rem:breg-mirror}
Formula (9) can be interpreted as one mirror-descent step for densities \citep{beck2003mirror}.  For a smooth strictly convex generator and the usual choice $s=\psi'(p)$ (right derivative at zero), the condition at $q^\star(z)>0$ in (9) is equivalent to the dual-coordinate form of the regularized solution
\[
\psi'(q^\star(z))=\psi'(p(z))-\frac{\ell(z)+\lambda^\star}{\alpha}.
\]
The regularized solution in dual coordinates is obtained by starting from $\psi'(p)$, taking a step of size $1/\alpha$ in the direction $-\ell$, and shifting by $-\lambda^\star/\alpha$ to enforce $\int_{\mathcal{Z}}q^\star\,\mathrm{d}\mu=1$.  Applying $(\psi')^{-1}$ then returns to density coordinates.  The finite-dimensional version is the Bregman proximal map on the probability simplex $\Delta_n=\{q\in[0,\infty)^n:\sum_iq_i=1\}$.  Its Karush--Kuhn--Tucker condition is $\ell_i+\alpha[\psi'(q_i)-\psi'(p_i)]+\lambda=0$ when $q_i>0$.  When $\psi(0)<\infty$, the condition at $q_i=0$ is $\ell_i+\alpha[\psi'_+(0)-\psi'(p_i)]+\lambda\ge0$.
\end{remark}

\begin{remark}[Dependence on the reference measure]
\label{rem:breg-reference}
The reference measure is part of a Bregman penalty.  Changing it while keeping the probability measures, generator, and weight fixed can change the optimizer.  For example, with $\mu^{\prime}=a\mu$ for $a>0$, the same measures have densities $p^{\prime}=p/a$ and $q^{\prime}=q/a$, so the least-squares divergence computed relative to $\mu^{\prime}$ is $a^{-1}$ times that computed relative to $\mu$, and the objective is preserved only if the weight is changed to $\alpha^{\prime}=a\alpha$.
\end{remark}

\section{Derivations for the Bregman rows of Table~1}
\label{app:breg-examples}

The first corollary below states the finite-attainment conditions in the Bregman block of Table~1; its integrals $K_\beta(\lambda)$ and $K_{\mathrm B}(\lambda)$ are $K_\beta$ and $K_{\mathrm B}$ from (4), written as functions of the scalar.  Both corollaries are stated in the setting of Theorem~2, with a measurable real-valued loss and no integrability assumption under the prior.

\begin{corollary}[Exact criteria for the Bregman rows]
\label{cor:breg-rows-main}
The following statements hold.
\begin{enumerate}
\item For density power with $\beta>0$, put
\[
 q_\lambda
 =\left(p^\beta-\frac{\beta}{\alpha}(\ell+\lambda)\right)_+^{1/\beta}.
\]
A finite minimizer exists if and only if some $\lambda\in\mathbb R$ satisfies
\begin{equation}
\label{eq:power-attainment-main}
 \int q_\lambda\,\mathrm{d}\mu=1,
 \qquad
 K_\beta(\lambda):=
 \int\left|q_\lambda^{\beta+1}-p^{\beta+1}\right|\,\mathrm{d}\mu<\infty.
\end{equation}
This includes least squares at $\beta=1$, for which the second integral is $\int|q_\lambda-p|(q_\lambda+p)\,\mathrm{d}\mu$.

\item For Burg, assume $p>0$ $\mu$-almost everywhere and put
\[
 d_\lambda=1+\frac{p(\ell+\lambda)}{\alpha},
 \qquad q_\lambda=\frac{p}{d_\lambda}.
\]
A finite minimizer exists if and only if some $\lambda\in\mathbb R$ satisfies
\begin{equation}
\label{eq:burg-attainment-main}
 \begin{gathered}
 d_\lambda>0\quad\mu\text{-almost everywhere},
 \qquad \int q_\lambda\,\mathrm{d}\mu=1,\\
 K_{\mathrm B}(\lambda):=
 \int\left\{\log_+d_\lambda+
       \left(d_\lambda^{-1}-1\right)_+\right\}\,\mathrm{d}\mu<\infty.
 \end{gathered}
\end{equation}
\end{enumerate}
In each case the optimizer is $q_\lambda$, and both the optimizer and the scalar $\lambda$ are unique.  These criteria require no prior integrability or separate feasibility assumption.
\end{corollary}

The integral tests in \eqref{eq:power-attainment-main} and \eqref{eq:burg-attainment-main} are evaluated at the same scalar that normalizes the density.  The next result also gives tests on the loss before that scalar is found.  It separates finite feasibility, meaning that at least one probability density has finite absolute expected loss and finite discrepancy, from the tail control and normalization of the candidate $q_\lambda$.

\begin{corollary}[Feasibility, tails, and scalar endpoints]
\label{cor:breg-feas-tail-main}
Write $X=\ell^-$.
\begin{enumerate}
\item\label{item:breg-power-tail}
For density power with $\beta>0$, define
\begin{equation}
\label{eq:power-feasibility-main}
 H_\beta=\int p\min\{|\ell|,p^\beta\}\,\mathrm{d}\mu.
\end{equation}
A finite feasible density exists exactly when $H_\beta<\infty$.  Finite attainment is equivalent to $H_\beta<\infty$, normalization of $q_\lambda$, and the tail condition
\begin{equation}
\label{eq:power-tail-main}
 \text{there exists $M>0$ such that}\qquad
 \int_{\{X>M\}}X^{1+1/\beta}\,\mathrm{d}\mu<\infty.
\end{equation}
Under \eqref{eq:power-tail-main}, put
\begin{equation}
\label{eq:power-endpoint-main}
 b_\mu=\sup\{t\in\mathbb R:\mu(\ell<t)<\infty\}\in\mathbb R\cup\{\infty\}.
\end{equation}
If $b_\mu=\infty$, a unique normalizer always exists.  If $b_\mu<\infty$, define
\[
 B_\beta=\int
 \left(p^\beta+\frac{\beta}{\alpha}(b_\mu-\ell)\right)_+^{1/\beta}
 \,\mathrm{d}\mu.
\]
Normalization is possible exactly when $B_\beta\ge1$.  For $B_\beta=1$ the normalizer is $\lambda^\star=-b_\mu$; for $B_\beta>1$ it is the unique root in $(-b_\mu,\infty)$.

If $\mu(\mathcal Z)<\infty$, finite attainment therefore reduces to
\[
 H_\beta<\infty,
 \qquad X\in L^{1+1/\beta}(\mu).
\]
If also $p\in L^{\beta+1}(\mu)$, only the displayed moment condition remains.

\item\label{item:breg-burg-tail}
For Burg, assume $p>0$ almost everywhere and define
\begin{equation}
\label{eq:burg-feasibility-main}
 H_{\mathrm B}=\int\log\left(1+\frac{p|\ell|}{\alpha}\right)\,\mathrm{d}\mu,
 \qquad
 H_{\mathrm B}^+=\int\log\left(1+\frac{p\ell^+}{\alpha}\right)\,\mathrm{d}\mu.
\end{equation}
A finite feasible density exists exactly when $H_{\mathrm B}<\infty$.  Put
\[
 h=\ell+\frac{\alpha}{p},
 \qquad a=\operatorname*{ess\,inf}_\mu h.
\]
Finite attainment requires $a> -\infty$.  When $a> -\infty$, put
\begin{equation}
\label{eq:burg-endpoint-main}
 B=\int\frac{\alpha}{h-a}\,\mathrm{d}\mu,
 \qquad
 I_1=\int\frac{X}{1+h-a}\,\mathrm{d}\mu,
 \qquad
 I_0=\int\frac{X}{h-a}\,\mathrm{d}\mu.
\end{equation}
Here $1/0=\infty$, and $I_0$ is used only when $B=1$, which ensures $h>a$ almost everywhere.  The exact criterion for finite attainment is
\begin{equation}
\label{eq:burg-tail-criterion-main}
 \begin{gathered}
 a> -\infty,\qquad H_{\mathrm B}^+<\infty,\\
 \bigl[B>1\text{ and }I_1<\infty\bigr]
 \ \text{or}\ 
 \bigl[B=1\text{ and }I_0<\infty\bigr],
 \end{gathered}
\end{equation}
where the two bracketed alternatives share the first two conditions.  In the first alternative, $\lambda^\star$ is the unique root of $\int\alpha/(h+\lambda)\,\mathrm{d}\mu=1$ in $(-a,\infty)$; in the second, $\lambda^\star=-a$.  If $H_{\mathrm B}<\infty$ and $a=-\infty$, the infimum is $-\infty$.
\end{enumerate}
\end{corollary}

For Burg, it is the effective loss $h=\ell+\alpha/p$ that must be essentially bounded below, rather than $\ell$ itself, so losses unbounded below are allowed.  The remaining weighted integral is necessary: the examples later in this subsection give normalized candidates whose loss and discrepancy are both infinite, as well as finite optimizers for unbounded losses.  The rest of this subsection proves the two corollaries.

\paragraph{Density power: the candidate and its finiteness test.}
Use $\psi_\beta(t)=t^{\beta+1}/\{\beta(\beta+1)\}$ for $t\ge0$ and $\beta>0$.  Then $\psi'_\beta(t)=t^\beta/\beta$, with boundary slopes $0$ and $\infty$, so $T_{\psi_\beta}(y)=(\beta y)_+^{1/\beta}$ for every real $y$.  Substitution into (9) gives the candidate $q_\lambda$ in Corollary~\ref{cor:breg-rows-main}; taking $\beta=1$ gives least squares.

Write $D_\beta(q,p)$ for the density-power integrand and put $\delta=\ell+\lambda$.  Direct substitution, including the case $q_\lambda=0$, gives
\begin{equation}
\label{eq:power-pointwise-value}
 \delta q_\lambda+\alpha D_\beta(q_\lambda,p)
 =\frac{\alpha}{\beta+1}
       \left(p^{\beta+1}-q_\lambda^{\beta+1}\right).
\end{equation}
On $\{\delta\ge0\}$, both terms on the left are nonnegative, so their separate integrability is exactly integrability of the power difference.  On $\{\delta<0\}$, $q_\lambda\ge p$ and
\[
 q_\lambda^\beta-p^\beta=\frac{\beta}{\alpha}\delta^-.
\]
For $q\ge p\ge0$, the elementary inequalities
\[
 (q^\beta-p^\beta)q
 \le q^{\beta+1}-p^{\beta+1}
 \le\frac{\beta+1}{\beta}(q^\beta-p^\beta)q
\]
follow by subtraction for the first bound and by integrating
$t^\beta\le q t^{\beta-1}$ over $(p,q)$ for the second.  Consequently
\begin{equation}
\label{eq:power-negative-comparison}
 \frac{\beta}{\alpha}\delta^-q_\lambda
 \le q_\lambda^{\beta+1}-p^{\beta+1}
 \le\frac{\beta+1}{\alpha}\delta^-q_\lambda.
\end{equation}
Moreover, \eqref{eq:power-pointwise-value} gives
$0\le\alpha D_\beta(q_\lambda,p)\le\delta^-q_\lambda$ there.
For a normalized candidate, replacing $|\delta|$ by $|\ell|$ changes the relevant loss integral by at most $|\lambda|$.  Equations~\eqref{eq:power-pointwise-value} and~\eqref{eq:power-negative-comparison} therefore show that \eqref{eq:power-attainment-main} is equivalent to separate finiteness of loss and discrepancy.  Theorem~2 proves Corollary~\ref{cor:breg-rows-main}(i).  Strict convexity gives a unique density; on its positive-mass support the equation
\[
 \lambda=-\ell-\frac{\alpha}{\beta}(q_\lambda^\beta-p^\beta)
\]
also makes the multiplier unique.

\paragraph{Density power: finite feasibility.}
We prove Corollary~\ref{cor:breg-feas-tail-main}\textup{(\ref{item:breg-power-tail})}.  Minimize the nonnegative integrand
\[
 |\ell|t+\alpha D_\beta(t,p)
\]
pointwise.  Its minimizer is
\[
 v=\left(p^\beta-\frac{\beta|\ell|}{\alpha}\right)_+^{1/\beta}\le p,
\]
and its minimum is
\[
 m_\beta=\frac{\alpha}{\beta+1}
                 (p^{\beta+1}-v^{\beta+1}).
\]
For $p>0$, set $x=\beta|\ell|/(\alpha p^\beta)$.  The bounds
\[
 \min\{x,1\}
 \le 1-(1-x)_+^{1+1/\beta}
 \le (1+1/\beta)\min\{x,1\}
 \qquad(x\ge0)
\]
show that $m_\beta$ is bounded above and below by positive constant multiples of $p\min\{|\ell|,p^\beta\}$; the same assertion holds at $p=0$ because both expressions vanish.  Thus $\int m_\beta\,\mathrm{d}\mu<\infty$ exactly when $H_\beta<\infty$.  A finite feasible density must dominate this pointwise minimum in its nonnegative absolute-loss objective, proving necessity.

Conversely, suppose $H_\beta<\infty$.  The function $v$ has mass at most one and finite integrals of both $|\ell|v$ and $D_\beta(v,p)$.  If its mass is less than one, choose a set $A$ of finite positive $\mu$-measure on which $p$ and $|\ell|$ are bounded.  Such a set exists by $\sigma$-finiteness and almost-everywhere finiteness.  Adding the missing mass as a constant density on $A$ preserves finiteness: the loss is bounded there, and the power discrepancy is bounded when both densities are bounded.  This constructs a finite feasible probability density.

\paragraph{Density power: the negative tail.}
If a finite optimizer exists, it is a normalized $q_\lambda$ and has finite negative expected loss.  On $\{X>2|\lambda|\}$,
\[
 q_\lambda\ge
 \left(\frac{\beta X}{2\alpha}\right)^{1/\beta}.
\]
This proves \eqref{eq:power-tail-main}; finite feasibility also gives $H_\beta<\infty$.

Conversely, assume $H_\beta<\infty$, \eqref{eq:power-tail-main}, and normalization at $\lambda$.  First, $\int Xp\,\mathrm{d}\mu<\infty$.  Indeed, on $\{X\le M\}$ the integrand $Xp$ is integrable because $p$ has unit mass; on $\{X>M,\ X\le p^\beta\}$ it is part of $H_\beta$, whereas on $\{X>M,\ X>p^\beta\}$ it is at most $X^{1+1/\beta}$.  For sufficiently large $X$, the power-sum inequality gives
\[
 q_\lambda\le C\bigl(p+X^{1/\beta}\bigr)
\]
for a constant depending on $\alpha,\beta,\lambda$.  It follows that $\int Xq_\lambda\,\mathrm{d}\mu<\infty$: the large-$X$ part uses the two integrable bounds just obtained, and the bounded-$X$ part uses normalization.  The finite comparison density supplied by $H_\beta<\infty$ now permits application of Theorem~2.

\paragraph{Density power: the mass endpoint.}
For each $\lambda$, the same power-sum inequality, together with $\int p\,\mathrm{d}\mu=1$, shows that
\begin{equation}
\label{eq:power-finite-mass-domain}
 G_\beta(\lambda):=\int q_\lambda\,\mathrm{d}\mu<\infty
 \quad\Longleftrightarrow\quad
 \int(-\ell-\lambda)_+^{1/\beta}\,\mathrm{d}\mu<\infty.
\end{equation}
Condition~\eqref{eq:power-tail-main} implies that some strict sublevel of $\ell$ has finite measure, so $b_\mu> -\infty$.  If $t<b_\mu$, then $\mu(\ell<t)<\infty$.  In the integral of $(t-\ell)_+^{1/\beta}$, the bounded-loss part is therefore finite, and the sufficiently negative part is finite by \eqref{eq:power-tail-main}.  If $t>b_\mu$, choose $s\in(b_\mu,t)$.  Then $\mu(\ell<s)=\infty$, while $(t-\ell)_+^{1/\beta}\ge(t-s)^{1/\beta}$ on that set, so the integral is infinite.

Consequently $G_\beta$ is finite on $(-b_\mu,\infty)$ and infinite on $(-\infty,-b_\mu)$ when $b_\mu$ is finite.  On each compact interval of finite masses, the leftmost candidate is an integrable envelope; dominated convergence proves continuity.  The mass decreases strictly wherever positive and tends to zero at $+\infty$.  Monotone convergence gives the limit $B_\beta$ as $\lambda\downarrow-b_\mu$, including the endpoint candidate itself when that integral is finite.  These facts prove the criterion $B_\beta\ge1$ and both assertions about its root.  If $b_\mu=\infty$, the finite-mass interval is all of $\mathbb R$, and monotone convergence gives mass tending to infinity as $\lambda\to-\infty$; hence there is always a unique root.

When $\mu$ is finite, $b_\mu=\infty$ and \eqref{eq:power-tail-main} is equivalent to $X\in L^{1+1/\beta}(\mu)$.  Finally $H_\beta\le\int p^{\beta+1}\,\mathrm{d}\mu$, which proves the asserted simplification when the prior density has this power moment.

\paragraph{Finite reference measure does not ensure finite feasibility.}
For least squares, take $\alpha=1$, Lebesgue measure on $(0,1)$, and
\[
 p(x)=\frac{1}{2\sqrt{x}},\qquad \ell(x)=p(x).
\]
Then $p$ is a probability density, but $H_1=\int_0^1p^2\,\mathrm{d}x=\infty$, so there is no finite feasible density.  Nevertheless the candidate is the constant $q_\lambda=(-\lambda)_+$, and it normalizes at $\lambda=-1$, giving $q=1$.  Its expected loss is finite, but its least-squares discrepancy is infinite.  This example shows why a finite reference measure alone cannot give finite attainment when prior integrability is removed.

\paragraph{An infinite-reference example without a normalizer.}
Take counting measure on the positive integers, $p(1)=1$, $\ell(1)=c>0$, and $p(n)=\ell(n)=0$ for $n\ge2$.  Both $H_\beta<\infty$ and \eqref{eq:power-tail-main} hold, but $b_\mu=0$ and
\[
 B_\beta=(1-\beta c/\alpha)_+^{1/\beta}<1.
\]
Thus finite attainment fails.  Directly, every $\lambda<0$ gives the same positive candidate at infinitely many tail coordinates and hence infinite mass; every $\lambda\ge0$ gives mass at most $B_\beta$.

\paragraph{Burg: the candidate and its finiteness test.}
For $\psi(t)=-\log t$ on $(0,\infty)$, the derivative is $-1/t$, with boundary slopes $-\infty$ and $0$.  Hence $T_\psi(y)=-1/y$ for $y<0$.  Substitution of $y=-1/p-(\ell+\lambda)/\alpha$ gives $q_\lambda=p/d_\lambda$, defined exactly when $d_\lambda>0$ almost everywhere.

Put $\delta=\ell+\lambda$ and $v=q_\lambda/p=d_\lambda^{-1}$.  The scalar equation gives
\begin{equation}
\label{eq:burg-pointwise-value}
 \delta q_\lambda=\alpha(1-v),
 \qquad
 \delta q_\lambda+\alpha(v-\log v-1)=-\alpha\log v.
\end{equation}
On $\{v\le1\}$, the two terms in the second left side are nonnegative, so their finiteness is exactly integrability of $-\log v=\log_+d_\lambda$.  On $\{v>1\}$, the negative loss for $\delta$ has magnitude $\alpha(v-1)$, and
\[
 0\le v-\log v-1\le v-1.
\]
Since a normalized density makes shifts of the loss integrable, these observations prove that \eqref{eq:burg-attainment-main} is equivalent to separate finite loss and discrepancy.  Theorem~2 then proves Corollary~\ref{cor:breg-rows-main}(ii).  Strict convexity gives density uniqueness, and $\lambda=-\ell-\alpha/p+\alpha/q_\lambda$ gives multiplier uniqueness.  When $\mu=P$, $p=1$ and the Burg penalty reduces to $\mathsf{KL}(P\Vert Q)$, recovering the reverse-Kullback--Leibler row.

\paragraph{Burg: finite feasibility and the positive tail.}
To prove Corollary~\ref{cor:breg-feas-tail-main}\textup{(\ref{item:breg-burg-tail})}, first minimize
\[
 |\ell|t+\alpha\{t/p-\log(t/p)-1\}
\]
pointwise.  Its minimizer is $v=p/(1+p|\ell|/\alpha)\le p$, and its minimum is $\alpha\log(1+p|\ell|/\alpha)$.  Necessity of $H_{\mathrm B}<\infty$ follows.  Conversely this condition makes $v$ a subdensity with finite absolute loss and discrepancy.  Add any missing mass on a set of finite positive measure where $p$ is bounded above and away from zero and $|\ell|$ is bounded.  On this set both $v$ and the modified density are bounded above and away from zero, so the change in the discrepancy is integrable.  This proves the feasibility characterization.

For a normalized candidate, \eqref{eq:burg-pointwise-value} shows that its finiteness test is equivalently
\begin{equation}
\label{eq:burg-positive-negative-test}
 H_{\mathrm B}^+<\infty,
 \qquad \int Xq_\lambda\,\mathrm{d}\mu<\infty.
\end{equation}
Indeed, the positive logarithmic term using $\ell+\lambda$ differs from that using $\ell$ by at most $|\lambda|p/\alpha$ pointwise, because $\log(1+x)$ is $1$-Lipschitz for $x\ge0$.  The negative-loss terms differ by at most $|\lambda|q_\lambda$.

\paragraph{Burg: the effective-loss endpoint.}
The candidate is $q_\lambda=\alpha/(h+\lambda)$.  Its domain requires $a> -\infty$ and $\lambda\ge-a$, with strict positivity of the denominator almost everywhere.  Suppose first that $I_1<\infty$.  This implies $\int Xp\,\mathrm{d}\mu<\infty$: on $\{X>\max\{1-a,0\}\}$ one has
\[
 p(1+h-a)=\alpha+p(1-X-a)\le\alpha,
 \qquad
 Xp\le\frac{\alpha X}{1+h-a},
\]
and bounded $X$ is integrable under $P$.

For $\lambda>-a$, let $A_\lambda=\{p(\ell+\lambda)<-\alpha/2\}$.  The preceding negative prior moment implies
\[
 \mu(A_\lambda)
 \le\frac2\alpha\int p(\ell+\lambda)^-\,\mathrm{d}\mu<\infty.
\]
Outside $A_\lambda$, $q_\lambda\le2p$; on $A_\lambda$, $q_\lambda\le\alpha/(a+\lambda)$.  Thus every interior candidate has finite mass.  Domination on compact multiplier intervals gives continuity, and the mass decreases strictly to zero at infinity.  Monotone convergence gives its left endpoint limit $B$.  Therefore $B>1$ gives a unique interior normalizer, while $B=1$ gives the endpoint normalizer.

For any interior $\lambda$, the two positive functions $h+\lambda$ and $1+h-a$ have a ratio bounded above and below by positive constants.  Thus $\int Xq_\lambda\,\mathrm{d}\mu<\infty$ is equivalent to $I_1<\infty$.  At the endpoint, normalization requires $B=1$, which implies $h>a$ almost everywhere, and the negative-loss condition is exactly $I_0<\infty$.  Notice that $I_0<\infty$ implies $I_1<\infty$ in this case.  Together with \eqref{eq:burg-positive-negative-test}, these observations prove sufficiency in \eqref{eq:burg-tail-criterion-main}.

Conversely, a finite optimizer is a normalized candidate with $\lambda\ge-a$ and satisfies \eqref{eq:burg-positive-negative-test}.  If $\lambda>-a$, the comparison of denominators gives $I_1<\infty$, and strict monotonicity forces $B>1$.  If $\lambda=-a$, normalization gives $B=1$ and finite negative loss gives $I_0<\infty$.  This proves necessity without assuming any prior integrability.

\paragraph{Burg with an effective loss unbounded below.}
Suppose $H_{\mathrm B}<\infty$ and $a=-\infty$.  The feasibility result supplies a density $q_0>0$ with finite absolute loss and discrepancy.  Choose a set $A$ of finite positive measure on which $p$ and $q_0$ are bounded above and away from zero and $|\ell|$ is bounded.  Fix $0<\varepsilon<\int_Aq_0\,\mathrm{d}\mu$ and remove mass $\varepsilon$ from $A$ by multiplying $q_0$ there by a fixed factor in $(0,1)$.  This incurs a fixed finite change $C$ in the objective.

For sufficiently large $n$, choose $B_n\subseteq\{h<-n\}$ disjoint from $A$, of finite positive measure, with $p,q_0$ bounded above and away from zero and $|\ell|$ bounded there.  Such subsets exist by truncation and $\sigma$-finiteness.  Put $v_n=\mathbf1_{B_n}/\mu(B_n)$ and add $\varepsilon v_n$ on $B_n$, leaving $q_0$ unchanged elsewhere.  These densities have finite objectives.  On $B_n$, the change in the Burg integrand is
\[
 \frac{\varepsilon v_n}{p}
 -\log\left(1+\frac{\varepsilon v_n}{q_0}\right)
 \le\frac{\varepsilon v_n}{p}.
\]
The total objective change is consequently at most
\[
 C+\varepsilon\int h v_n\,\mathrm{d}\mu\le C-\varepsilon n.
\]
This proves the infimum is $-\infty$.  Local modification is useful here: on an infinite reference measure, globally multiplying a density by a constant can produce infinite Burg discrepancy.

\paragraph{Unbounded losses: attainment and normalization without finite attainment.}
Take Lebesgue measure on $(0,1)$, $\alpha=1$, and, for $\gamma>0$, let
\begin{equation}
\label{eq:burg-power-prior-example}
 p_\gamma(x)=(\gamma+1)x^\gamma,
 \qquad
 \ell_\gamma(x)=1-\frac1{p_\gamma(x)}.
\end{equation}
The loss is unbounded below but belongs to $L^1(P)$, since
$|\ell_\gamma|p_\gamma=|p_\gamma-1|$.  The effective loss is identically one, so the unique normalizer is $\lambda=0$ and its candidate is $q=1$.  This is a finite optimizer exactly when $0<\gamma<1$, because $\int_0^1p_\gamma^{-1}\,\mathrm{d}x<\infty$ exactly in that range.

For every $\gamma>0$, the infimum is
\[
 \log(\gamma+1)-\gamma.
\]
To see this, for every finite competitor the separately finite terms may be combined to give
\[
 J(q)=\int_0^1\log p_\gamma\,\mathrm{d}x
       -\int_0^1\log q\,\mathrm{d}x
 \ge\int_0^1\log p_\gamma\,\mathrm{d}x
 =\log(\gamma+1)-\gamma,
\]
by Jensen's inequality, with equality only at $q=1$.  The logarithms are integrable for such a competitor: its finite loss makes $\int q/p_\gamma$ finite, and finite Burg discrepancy then makes $\log(q/p_\gamma)$ integrable; also $\log p_\gamma\in L^1(0,1)$.  The finite densities
\[
 q_\varepsilon(x)=
 \frac{p_\gamma(x)\mathbf1_{(0,\varepsilon)}(x)
       +\mathbf1_{[\varepsilon,1)}(x)}
      {1-\varepsilon+\varepsilon^{\gamma+1}}
\]
have $\int\log q_\varepsilon\to0$, so their objectives approach this lower bound.  For $\gamma\ge1$, the normalized candidate $q=1$ has infinite negative loss and infinite discrepancy, and the finite infimum is unattained.  There is no competitor with finite discrepancy and expected loss $-\infty$: on this finite reference measure, $v-\log v-1\ge v/2-C$ for a constant $C$, so finite Burg discrepancy forces $\int q/p_\gamma<\infty$.

\section{Proof of Proposition~3}
\label{app:facts}

Let $\varphi$ denote the generator and let $D=\{y\in\mathbb{R}:S_\varphi(y)\ne\emptyset\}$.  By strict convexity, $S_\varphi(y)$ consists of a single point, denoted $T(y)$, for every $y\in D$.  We first show that $D$ is an interval, that $T$ is continuous and nondecreasing on $D$, and that $T(y)<T(y')$ whenever $y<y'$ in $D$ and $T(y)>0$.

Since $\varphi$ is strictly convex and differentiable on $(0,\infty)$, its derivative is continuous and strictly increasing there, with range the open interval $R=(\varphi'_+(0),\varphi'_-(\infty))$, where $\varphi'_+(0)=\lim_{t\downarrow0}\varphi'(t)$ and $\varphi'_-(\infty)=\lim_{t\uparrow\infty}\varphi'(t)$.  A point $t>0$ minimizes $\varphi(t)-yt$ over $[0,\infty)$ if and only if $\varphi'(t)=y$.  Hence $R\subseteq D$, and $T=(\varphi')^{-1}$ on $R$ is positive, continuous and strictly increasing.  For $y\in D\setminus R$ the minimizer is $T(y)=0$, that is, $\varphi(t)-\varphi(0)\ge yt$ for all $t>0$.  Suppose this holds for some $y_0$.  Letting $t\downarrow0$ shows that $\varphi$ is finite and continuous at zero.  Convexity gives $y_0\le\{\varphi(t)-\varphi(0)\}/t\le\varphi'(t)$ for all $t>0$, so $\varphi'_+(0)\ge y_0$ is finite, and
\[
 \varphi(t)-\varphi(0)=\int_0^t\varphi'(x)\,\mathrm{d}x\ge t\,\varphi'_+(0)
\]
shows that $T(y)=0$ for every $y\le\varphi'_+(0)$.  Thus $D$ is either $R$ or $(-\infty,\varphi'_-(\infty))$.  In the second case $T$ vanishes on $(-\infty,\varphi'_+(0)]$ and is continuous at $\varphi'_+(0)$, because $(\varphi')^{-1}(y)\downarrow0$ as $y\downarrow\varphi'_+(0)$.  This proves the three claims about $T$.

Now $v_\lambda(z)=T(b(z)-\lambda/\alpha)$, where $b=-\ell/\alpha$ in the $f$-divergence case and $b=s-\ell/\alpha$ in the Bregman case.  It is defined at $z$ exactly when $b(z)-\lambda/\alpha\in D$, and it is measurable because $T$ is continuous.  Let $\lambda_1<\lambda_2$ in $\Lambda$.  The properties of $T$ give $v_{\lambda_1}\ge v_{\lambda_2}$ almost everywhere, with strict inequality wherever $v_{\lambda_2}>0$.  Hence $G$ is nonincreasing.  If $G(\lambda_2)>0$, then $\{v_{\lambda_2}>0\}$ has positive measure and, both masses being finite, $G(\lambda_1)>G(\lambda_2)$.  So $G$ is strictly decreasing wherever it is positive.

For continuity, fix $\lambda\in\Lambda$ and choose $\delta>0$ with $[\lambda-\delta,\lambda+\delta]\subset\Lambda$.  Because $D$ is an interval, almost every $z$ satisfies $b(z)-(\lambda+\eta)/\alpha\in D$ for all $|\eta|\le\delta$ simultaneously, and for such $z$,
\[
 0\le v_{\lambda+\eta}(z)\le v_{\lambda-\delta}(z).
\]
The right side is integrable because $G(\lambda-\delta)<\infty$, and $v_{\lambda+\eta}\to v_\lambda$ pointwise as $\eta\to0$ by continuity of $T$.  Dominated convergence gives $G(\lambda+\eta)\to G(\lambda)$.

Finally, under the endpoint condition, continuity and the intermediate value theorem give some $\lambda^\star\in\Lambda$ with $G(\lambda^\star)=1$, and strict decrease at positive levels makes it unique.  The argument permits a flat region where $G=0$, as occurs for the truncated densities.

\section{Proofs and consequences for R\'enyi orders greater than one}
\label{app:renyi-proofs}

This section proves Theorem~4, determines the value of the infimum when finite attainment fails, and collects consequences for the lower tail of the loss and for the computation of $\lambda^\star$.

\paragraph{Notation and the value of the infimum.}
Fix $r>1$.  Write $h=r/(r-1)$, $s=h-1=1/(r-1)$, $c=\alpha h$, $X=\ell^-$, and denote integration under $P$ by $E$.  Let $\mathcal U(P)$ denote the set of measurable functions $u:\mathcal Z\to[0,\infty)$ with $\int_{\mathcal Z}u\,\mathrm dP=1$, and write
\begin{equation}
\label{eq:renyi-primal-main}
  V_r=\inf_{u\in\mathcal U(P)}J_r(u),
  \qquad
  J_r(u)=\int_{\mathcal Z}\ell u\,\mathrm dP
  +\frac{\alpha}{r-1}\log\int_{\mathcal Z}u^r\,\mathrm dP,
\end{equation}
for the value of (1) and its objective.  The infimum can be computed over densities for which both terms are finite.  Such densities always exist, although the infimum can equal $-\infty$.  The candidate solutions are the truncated power densities
\begin{equation}
\label{eq:renyi-density-main}
  u_\lambda(z)=\frac{(\lambda-\ell(z))_+^s}{A_s(\lambda)},
  \qquad \lambda>\ell_\star,
\end{equation}
defined whenever $A_s(\lambda)<\infty$.  Each candidate already has unit mass.  Alongside the profile $\Phi_r$ of (13), we use its shifted version
\begin{equation}
\label{eq:renyi-shifted-profile}
  \widetilde\Phi_r(\lambda)
  =\Phi_r(\lambda)-c+c\log c
  =\lambda-c-\alpha\log\left\{\frac{A_h(\lambda)}{c^h}\right\},
  \qquad \lambda>\ell_\star.
\end{equation}
The two profiles differ by a constant, so they have the same minimizers and the same stationary points; the shift is chosen so that the infimum of $\widetilde\Phi_r$ is $V_r$.  Since $(X-t)_+=(-t-\ell)_+$ for $t>0$, the constant in (14) can also be written as
\[
 T_{r,\alpha}=\limsup_{t\to\infty}e^{t/\alpha}E(X-t)_+^h.
\]
It measures the limiting contribution of losses tending to $-\infty$.  Finally, put
\[
 K_\alpha=\sup_{t\ge0}e^{t/\alpha}P(X>t),
 \qquad
 \overline L_\alpha=\limsup_{t\to\infty}e^{t/\alpha}P(X>t).
\]

\begin{proposition}[Value of the infimum for $r>1$]
\label{prop:renyi-value}
Let $r>1$.
\begin{enumerate}
\item[\textup{(i)}] The infimum $V_r$ is finite if and only if $K_\alpha<\infty$, equivalently if and only if $T_{r,\alpha}<\infty$.  If these conditions fail, $V_r=-\infty$.  In particular, if $T_{r,\alpha}<\infty$ but (15) fails, the infimum is finite and unattained.
\item[\textup{(ii)}] If $EX^h<\infty$, then $\Phi_r$ is finite and continuously differentiable on $(\ell_\star,\infty)$,
\begin{equation}
\label{eq:renyi-value-main}
  V_r=\inf_{\lambda>\ell_\star}\widetilde\Phi_r(\lambda),
\end{equation}
and every $\lambda>\ell_\star$ satisfies
\[
 \widetilde\Phi_r(\lambda)\ge J_r(u_\lambda)\ge V_r,
\]
with equality in the first inequality exactly when $\lambda$ solves (16).  Thus every trial value $\lambda>\ell_\star$ gives a density with finite objective and an upper bound on that objective.  If $EX^h=\infty$, then $V_r=-\infty$.
\item[\textup{(iii)}] The constant $T_{r,\alpha}$ satisfies
\begin{equation}
\label{eq:renyi-tail-limsup-bounds}
 (\alpha h/e)^h\overline L_\alpha
 \le T_{r,\alpha}
 \le\Gamma(h+1)\alpha^h\overline L_\alpha.
\end{equation}
In particular,
\begin{equation}
\label{eq:renyi-vanishing-tail}
 T_{r,\alpha}=0
 \quad\Longleftrightarrow\quad
 e^{t/\alpha}P(\ell<-t)\longrightarrow0
 \qquad(t\to\infty).
\end{equation}
\item[\textup{(iv)}] If $\ell_\star>-\infty$ or $Z_\alpha=\int_{\mathcal Z}e^{-\ell/\alpha}\,\mathrm dP<\infty$, then $T_{r,\alpha}=0$ and $\Phi_r(\lambda)\to\infty$ at both ends of $(\ell_\star,\infty)$.  If $\ell_\star>-\infty$, then moreover $\Phi_r(\lambda)\ge\lambda-c\log(\lambda-\ell_\star)$ for every $\lambda>\ell_\star$.
\end{enumerate}
\end{proposition}

The rest of this section proves Theorem~4 and Proposition~\ref{prop:renyi-value} together.  The proofs first reduce the value $V_r$ to the scalar profile, then identify when that value is finite and attained.

\paragraph{The scalar value formula.}
Suppose first that $EX^h<\infty$.  Every density with $S(u)=Eu^r<\infty$ satisfies $E(Xu)<\infty$ by H\"older's inequality.  Thus its loss expectation exists in $(-\infty,\infty]$.  For $\lambda>\ell_\star$, the moments $A_s(\lambda)$ and $A_h(\lambda)$ are finite and positive.  Put $m_u=E(\ell u)$.  Given any $u$ with finite objective, choose $\lambda=m_u+c>\ell_\star$.  The unit-mass constraint and H\"older's inequality give
\[
 c=E[u(\lambda-\ell)]
 \le E[u(\lambda-\ell)_+]
 \le S(u)^{1/r}A_h(\lambda)^{1/h}.
\]
Taking logarithms yields $J_r(u)\ge\widetilde\Phi_r(m_u+c)$.  Equality holds precisely when $u=u_{m_u+c}$ almost surely: equality in H\"older's inequality makes $u$ proportional to $(m_u+c-\ell)_+^s$, and this form also gives equality in the first inequality.

Conversely, every $u_\lambda$ has finite objective.  Its positive-loss integral is finite because it is supported on $\{\ell<\lambda\}$, and H\"older's inequality bounds its negative-loss integral.  Substitution gives
\[
 m_{u_\lambda}=\lambda-\frac{A_h(\lambda)}{A_s(\lambda)},
 \qquad
 S(u_\lambda)=\frac{A_h(\lambda)}{A_s(\lambda)^r}.
\]
Writing $x_\lambda=A_h(\lambda)/(cA_s(\lambda))>0$, we obtain
\begin{equation}
\label{eq:renyi-profile-remainder}
 \widetilde\Phi_r(\lambda)-J_r(u_\lambda)
 =c\{x_\lambda-1-\log x_\lambda\}\ge0.
\end{equation}
Equality holds exactly when $x_\lambda=1$, that is, when $\lambda$ solves (16).  The two inequalities prove \eqref{eq:renyi-value-main}.  Any scalar minimizer gives an optimal density by \eqref{eq:renyi-profile-remainder}.  At a finite primal minimizer the lower bound must be an equality, so $m_u+c$ minimizes $\Phi_r$ and its density is $u$.  On compact threshold intervals, powers up to $h$ are dominated by a constant times $1+X^h$.  Dominated convergence therefore gives
\[
 A_h'=hA_s,
 \qquad
 \Phi_r'=1-cA_s/A_h,
\]
and hence (16).

If $EX^h=\infty$, choose a bounded density $b$ supported on a positive-probability set of bounded losses.  Put $Y_n=X\mathbf1_{\{X\le n\}}$ and, for sufficiently large $n$, set
\[
 v_n=\frac{Y_n^s}{\|Y_n^s\|_{L^r(P)}},
 \qquad
 u_n=v_n+(1-Ev_n)b.
\]
Then $Ev_n\le1$, each $u_n$ has mass one, and
\[
 \|u_n\|_{L^r(P)}\le1+\|b\|_{L^r(P)},
 \qquad
 E(Xv_n)=(EY_n^h)^{1/h}\longrightarrow\infty.
\]
Every $u_n$ has finite expected loss, these expectations tend to $-\infty$, and the divergences stay bounded.  Thus $V_r=-\infty$.

\paragraph{A finite infimum and the exponential tail bound.}
Define the auxiliary profile
\[
 F_{r,\alpha}(\lambda)=e^{-\lambda/\alpha}A_h(\lambda),
 \qquad
 M_{r,\alpha}=\sup_{\lambda>\ell_\star}F_{r,\alpha}(\lambda).
\]
When $EX^h<\infty$, the scalar value formula gives
\begin{equation}
\label{eq:renyi-exponential-profile-value}
 V_r=-c+c\log c-\alpha\log M_{r,\alpha}.
\end{equation}
If $EX^h=\infty$, every $A_h(\lambda)$ is infinite, so $M_{r,\alpha}=\infty$ and the same equality holds with value $-\infty$.

Suppose $K_\alpha<\infty$.  Integration of the tail gives
\[
 EX^h
 =h\int_0^\infty t^{h-1}P(X>t)\,\mathrm dt
 \le K_\alpha\Gamma(h+1)\alpha^h<\infty.
\]
For $t\ge0$,
\begin{equation}
\label{eq:renyi-tail-convolution}
\begin{aligned}
 F_{r,\alpha}(-t)
 &=h\int_0^\infty y^{h-1}e^{t/\alpha}P(X>t+y)\,\mathrm dy\\
 &=h\int_0^\infty y^{h-1}e^{-y/\alpha}
       \{e^{(t+y)/\alpha}P(X>t+y)\}\,\mathrm dy\\
 &\le K_\alpha\Gamma(h+1)\alpha^h.
\end{aligned}
\end{equation}
For $\lambda\ge0$,
\[
 A_h(\lambda)\le2^{h-1}(\lambda^h+EX^h),
\]
so $F_{r,\alpha}(\lambda)\to0$ as $\lambda\to\infty$.  It follows that $M_{r,\alpha}<\infty$.

Conversely, on $\{X>t\}$ the quantity $(c-t-\ell)_+$ is at least $c$.  Therefore
\[
 F_{r,\alpha}(c-t)
 \ge c^he^{-c/\alpha}e^{t/\alpha}P(X>t),
 \qquad
 M_{r,\alpha}\ge(\alpha h/e)^hK_\alpha.
\]
Whenever the right side of the first inequality is positive, $c-t>\ell_\star$, so the use of the supremum is legitimate.  Thus $M_{r,\alpha}<\infty$ if and only if $K_\alpha<\infty$.  Since $M_{r,\alpha}>0$, \eqref{eq:renyi-exponential-profile-value} proves the finite-value assertion in Proposition~\ref{prop:renyi-value}(i).

The condition $K_\alpha<\infty$ is also equivalent to $T_{r,\alpha}<\infty$.  If $K_\alpha<\infty$, then \eqref{eq:renyi-tail-convolution} gives $T_{r,\alpha}<\infty$.  For the converse, if $EX^h=\infty$, then $E(X-t)_+^h=\infty$ for every finite $t$, which would give $T_{r,\alpha}=\infty$.  Hence a finite $T_{r,\alpha}$ implies $EX^h<\infty$.  The profile is then continuous, tends to zero at $+\infty$, and is bounded on its far left by the definition of $T_{r,\alpha}$.  It is therefore bounded everywhere, giving $M_{r,\alpha}<\infty$.

The same estimates prove the bounds \eqref{eq:renyi-tail-limsup-bounds}.  Taking the limiting upper values in
$F_{r,\alpha}(c-t)\ge c^he^{-c/\alpha}e^{t/\alpha}P(X>t)$
gives the lower bound.  If $\overline L_\alpha<\infty$, then for every $\varepsilon>0$ the expression in braces in \eqref{eq:renyi-tail-convolution} is at most $\overline L_\alpha+\varepsilon$ for all sufficiently large $t$ and all $y\ge0$.  Integrating and then letting $\varepsilon\downarrow0$ gives the upper bound.  If $\overline L_\alpha=\infty$, the lower bound gives $T_{r,\alpha}=\infty$.  Since the scaled tail is nonnegative, these inequalities also prove \eqref{eq:renyi-vanishing-tail}.

\paragraph{The exact attainment condition.}
Assume $T_{r,\alpha}<\infty$.  If $\ell_\star>-\infty$, then
\[
 A_h(\lambda)\le(\lambda-\ell_\star)^h
 \qquad(\lambda>\ell_\star).
\]
Thus $F_{r,\alpha}$ tends to zero at both ends of its domain and has a positive maximum.  In this case $T_{r,\alpha}=0$, and (15) holds for every $b>\ell_\star$.  Taking logarithms in the same bound gives $\Phi_r(\lambda)\ge\lambda-c\log(\lambda-\ell_\star)$, whose right side tends to $\infty$ at both ends of $(\ell_\star,\infty)$; this is the part of Proposition~\ref{prop:renyi-value}(iv) for losses bounded below.

Suppose instead that $\ell_\star=-\infty$.  Then $F_{r,\alpha}$ is positive on the whole line, with
\[
 \limsup_{\lambda\to-\infty}F_{r,\alpha}(\lambda)=T_{r,\alpha},
 \qquad
 \lim_{\lambda\to\infty}F_{r,\alpha}(\lambda)=0.
\]
If some finite $b$ satisfies $F_{r,\alpha}(b)>T_{r,\alpha}$, both ends of the line lie below this value outside a compact interval; continuity gives a global maximum.  If $F_{r,\alpha}(b)=T_{r,\alpha}$, either this is already the global maximum or a larger value reduces to the preceding case.  Conversely, an attained maximum is at least the limiting upper value $T_{r,\alpha}$.  Since $F_{r,\alpha}=\exp(-\Phi_r/\alpha)$, the maximizers of $F_{r,\alpha}$ are the minimizers of $\Phi_r$.  This proves (15) and completes the proof of Theorem~4; the last assertion of Proposition~\ref{prop:renyi-value}(i) follows by combining the theorem with the finite-value assertion.

For the remaining part of Proposition~\ref{prop:renyi-value}(iv), $Z_\alpha<\infty$ first implies $EX^h<\infty$.  If $\ell_\star=-\infty$, then
\[
 e^{-\lambda/\alpha}A_h(\lambda)
 \le(\alpha h/e)^h
       E[e^{-\ell/\alpha}\mathbf1_{\{\ell<\lambda\}}]
 \longrightarrow0
 \qquad(\lambda\to-\infty).
\]
Consequently $T_{r,\alpha}=0$.  The profile $\Phi_r$ tends to $\infty$ at $+\infty$ by the polynomial bound on $A_h$, and at its left endpoint by this estimate, or by the bound using $\ell_\star$ when that endpoint is finite.

\paragraph{The critical regularization weight.}
For losses that are unbounded below, an exponential-moment test settles finite attainment for all weights except possibly one.

\begin{corollary}[Critical regularization weight]
\label{cor:renyi-critical-alpha}
Put
\[
 \theta_c=\sup\left\{\theta\ge0:
       \int_{\mathcal Z}e^{\theta X}\,\mathrm dP<\infty\right\},
 \qquad
 \alpha_c=\frac1{\theta_c},
\]
with $1/0=\infty$ and $1/\infty=0$.  For every $r>1$, finite attainment holds whenever $\alpha>\alpha_c$, and $V_r=-\infty$ whenever $0<\alpha<\alpha_c$.  If $0<\alpha_c<\infty$, only $\alpha=\alpha_c$ remains to be decided by (15); in this case, for each fixed $r>1$, the weights giving finite attainment form either $(\alpha_c,\infty)$ or $[\alpha_c,\infty)$.  If $\alpha_c=0$, all positive weights give finite attainment; if $\alpha_c=\infty$, none do.
\end{corollary}

\begin{proof}
The moment set defining $\theta_c$ is an interval.  If $\alpha>\alpha_c$, then $1/\alpha<\theta_c$ and $Ee^{X/\alpha}<\infty$, so $Z_\alpha<\infty$, and Proposition~\ref{prop:renyi-value}(iv) with Theorem~4 gives attainment.  If $0<\alpha<\alpha_c$ and $V_r$ were finite, then $K_\alpha<\infty$.  For every $0<\theta<1/\alpha$, the identity
\begin{equation}
\label{eq:renyi-exponential-moment}
 Ee^{\theta X}
 =1+\theta\int_0^\infty e^{\theta t}P(X>t)\,\mathrm dt
\end{equation}
would then give a finite moment.  This would imply $\theta_c\ge1/\alpha$, contradicting $\alpha<\alpha_c$.  The endpoint cases $\theta_c=0,\infty$ follow by the same argument.
\end{proof}

The critical weight is the same for every order $r>1$.  At that weight, the following two consequences often make the remaining comparison in (15) explicit.

\begin{corollary}[A limit for the scaled lower tail]
\label{cor:renyi-tail-limit}
Suppose the limit
\[
 L_\alpha=\lim_{t\to\infty}e^{t/\alpha}P(X>t)
\]
exists in $[0,\infty]$.  Then
\[
 T_{r,\alpha}=\Gamma(h+1)\alpha^hL_\alpha.
\]
If $L_\alpha=0$, finite attainment holds.  If $L_\alpha=\infty$, then $V_r=-\infty$.  If $0<L_\alpha<\infty$, then $V_r$ is finite, and attainment holds if and only if
\begin{equation}
\label{eq:renyi-tail-limit-main}
 \int_{\mathcal Z}(b-\ell)_+^h\,\mathrm dP
 \ge e^{b/\alpha}\Gamma(h+1)\alpha^hL_\alpha
 \quad\text{for some }b>\ell_\star.
\end{equation}
\end{corollary}

\begin{corollary}[An exponential tail with a polynomial factor]
\label{cor:renyi-polynomial-tail}
Suppose
\[
 P(X>t)\sim C t^\eta e^{-\theta t}
 \qquad(t\to\infty),
 \qquad C>0,\quad\theta>0,\quad\eta\in\mathbb R.
\]
Then $\alpha_c=1/\theta$.  At the critical weight, every order $r>1$ has finite attainment when $\eta<0$, and has infimum $-\infty$ when $\eta>0$.  When $\eta=0$, the infimum is finite and \eqref{eq:renyi-tail-limit-main}, with $\alpha=1/\theta$ and $L_\alpha=C$, is the exact attainment condition.
\end{corollary}

\begin{proof}[Proof of Corollaries~\ref{cor:renyi-tail-limit} and~\ref{cor:renyi-polynomial-tail}]
If $L_\alpha<\infty$, the expression in braces in \eqref{eq:renyi-tail-convolution} is uniformly bounded for all sufficiently large $t$ and all $y\ge0$.  Dominated convergence gives
\[
 \lim_{t\to\infty}F_{r,\alpha}(-t)
 =hL_\alpha\int_0^\infty y^{h-1}e^{-y/\alpha}\,\mathrm dy
 =\Gamma(h+1)\alpha^hL_\alpha.
\]
If $L_\alpha=\infty$, Fatou's lemma gives the same identity with value $\infty$.  Theorem~4 now proves Corollary~\ref{cor:renyi-tail-limit}.

Under the polynomial-factor assumption, \eqref{eq:renyi-exponential-moment} shows that the exponential-moment abscissa is $\theta_c=\theta$.  At $\alpha=1/\theta$, the scaled tail is asymptotic to $Ct^\eta$.  Its limit is zero, infinity, or $C$ according as $\eta<0$, $\eta>0$, or $\eta=0$.  Corollary~\ref{cor:renyi-tail-limit} proves the remaining assertions.  In particular, for $-1\le\eta<0$, the critical Gibbs normalizer is infinite but every R\'enyi order greater than one has a finite optimizer.
\end{proof}

At the critical weight, ordinary moment conditions alone cannot settle attainment.  In fact, bounded changes to the loss can change attainment while leaving its entire sufficiently low tail unchanged.  The next proposition gives such a pair, as well as examples in which $Z_\alpha=\infty$ but the infimum is finite and either attained or unattained.  The finite comparison in (15) retains precisely the information that a tail-integrability condition would omit.

\begin{proposition}[Critical tails and changes on bounded losses]
\label{prop:renyi-critical-tail}
Fix $\alpha=1$ and $r>1$, and let $h=r/(r-1)$.
\begin{enumerate}
\item If $\ell=-E$ for a standard exponential random variable $E$, then $Z_1=\infty$, but the GVI infimum is finite and attained.  Every threshold $\lambda\le0$ is optimal.
\item Let $Y=\max(E_1,E_2)$ for independent standard exponential random variables, and put $\ell_0=-Y$.  Then the GVI infimum is finite and unattained.
\item For all sufficiently large finite $M$, the loss $\ell_M=-\max(Y,M)$ has a finite optimizer.  It agrees with $\ell_0$ whenever $\ell_0\le-M$, and $|\ell_M-\ell_0|\le M$.
\end{enumerate}
All these losses have finite moments of every polynomial order, exponential-moment abscissa $\theta_c=1$, and an infinite Gibbs normalizer at $\alpha=1$.  The two losses in parts~(ii)--(iii) have identical lower tails below $-M$.
\end{proposition}
\begin{proof}
For a standard exponential variable and $\lambda\le0$,
\[
 A_h(\lambda)=e^\lambda\Gamma(h+1),
 \qquad F_{r,1}(\lambda)=\Gamma(h+1).
\]
For $\lambda>0$, integration by parts gives
\[
 A_h(\lambda)=\lambda^h+hA_{h-1}(\lambda),
 \qquad F_{r,1}'(\lambda)=-e^{-\lambda}\lambda^h<0.
\]
Thus every $\lambda\le0$ maximizes the profile, and
\[
 V_r=-h+h\log h-\log\Gamma(h+1).
\]

The density of $Y$ is $2e^{-y}-2e^{-2y}$ on $(0,\infty)$.  For $\lambda\le0$, its profile is
\[
 F_{r,1}^{(0)}(\lambda)
 =\Gamma(h+1)\{2-2^{-h}e^\lambda\}
 <2\Gamma(h+1).
\]
For $\lambda>0$, it is twice the exponential profile just considered, minus
\[
 2e^{-\lambda}\int_0^\infty(\lambda+y)^h e^{-2y}\,\mathrm dy>0.
\]
Hence it is still strictly below $2\Gamma(h+1)$.  Its supremum is $2\Gamma(h+1)$, approached as $\lambda\to-\infty$, and
\[
 V_r=-h+h\log h-\log\{2\Gamma(h+1)\}
\]
is unattained.

For $\ell_M$, take $b=h-M$.  Since $\max(Y,M)\ge M$,
\[
 F_{r,1}^{(M)}(h-M)
 \ge e^{M-h}h^h.
\]
This exceeds $2\Gamma(h+1)$ for sufficiently large $M$, whereas the limiting left-tail value is unchanged.  Theorem~4 gives finite attainment.  For $r=2$, the choice $M=4$ suffices:
$F_{2,1}^{(4)}(-2)\ge4e^2>4$.
\end{proof}

\paragraph{Stationary thresholds without attainment.}
\label{ex:renyi-stationary-unattained}
Let $Y$ be as in Proposition~\ref{prop:renyi-critical-tail}, put
$Z=Y\mathbf1_{\{Y>4\}}$, and take $\ell=-Z$, $\alpha=1$, and $r=2$.
Because $Z\le Y$, its profile satisfies
\[
 F_{2,1}^{Z}(\lambda)\le F_{2,1}^{Y}(\lambda)<4
 \qquad(\lambda\in\mathbb R).
\]
For $\lambda\le-4$ the profiles agree, because $(\lambda+Y)_+$ and $(\lambda+Z)_+$ both vanish on $\{Y\le4\}$.  Consequently $\sup_\lambda F_{2,1}^{Z}(\lambda)=4$ and no optimizer exists.  Nevertheless, write
\[
 a=EZ=10e^{-4}-\frac92e^{-8},
 \qquad
 b=EZ^2=52e^{-4}-\frac{41}{2}e^{-8}.
\]
For $\lambda\ge0$,
\[
 F_{2,1}^{Z}(\lambda)
 =e^{-\lambda}(\lambda^2+2a\lambda+b).
\]
It has the two positive stationary thresholds
\[
 \lambda_\pm=1-a\pm\sqrt{1+a^2-b}
 \simeq0.522618,\ 1.114088.
\]
The second is a local maximum of $F_{2,1}^{Z}$, with value approximately $0.850558$, and hence a local minimum of $\Phi_2$.  Both thresholds solve (16).  Thus even the existence of stationary thresholds, or of a local scalar minimum, does not imply GVI attainment.

\paragraph{Exact objective gap.}
Assume $K_\alpha<\infty$.  Recall that $m_u=\int_{\mathcal{Z}}\ell u\,\mathrm{d}P$ and $S(u)=\int_{\mathcal{Z}}u^r\,\mathrm{d}P$.  Fix a feasible $u$ with finite objective and put $\lambda_u=m_u+c>\ell_\star$.  For $\lambda>\ell_\star$, define
\[
  h_\lambda=(\lambda-\ell)_+,\qquad
  H_\lambda(u)=\int_{\mathcal{Z}}u h_\lambda\,\mathrm{d}P,\qquad
  \Lambda_u=\int_{\mathcal{Z}}u(\ell-\lambda_u)_+\,\mathrm{d}P.
\]
Then $H_{\lambda_u}(u)=c+\Lambda_u>0$.  Define probability measures $R_u$ and $R_\lambda$ by
\[
  \frac{\mathrm{d}R_u}{\mathrm{d}P}=\frac{u^r}{S(u)},\qquad
  \frac{\mathrm{d}R_\lambda}{\mathrm{d}P}=\frac{h_\lambda^{s+1}}{A_{s+1}(\lambda)}.
\]
For probability measures $R,T\ll P$ and $\eta=1/r\in(0,1)$, use
\[
  \mathsf{D}_{\mathrm{R}}^{(\eta)}(R\Vert T)
  =\frac{1}{\eta-1}\log\int_{\mathcal{Z}} \left(\frac{\mathrm{d}R}{\mathrm{d}P}(z)\right)^\eta\left(\frac{\mathrm{d}T}{\mathrm{d}P}(z)\right)^{1-\eta}\,P(\mathrm{d}z).
\]
This auxiliary R\'enyi divergence uses the general definition for orders below one and does not require $R\ll T$.  The divergence is nonnegative by H\"older's inequality.  Substitution of these densities yields, whenever $H_\lambda(u)>0$,
\[
  \frac{\alpha}{r-1}\log S(u)
  =c\log H_\lambda(u)-\alpha\log A_{s+1}(\lambda)
   +\alpha\mathsf{D}_{\mathrm{R}}^{(1/r)}(R_u\Vert R_\lambda).
\]
At $\lambda=\lambda_u$, the identity $H_{\lambda_u}(u)=c+\Lambda_u>0$ also ensures that the R\'enyi divergence between $R_u$ and $R_{\lambda_u}$ is finite.  Thus, since $V_r=\inf\widetilde\Phi_r$ by \eqref{eq:renyi-value-main}, every feasible $u$ with finite objective satisfies
\begin{equation}
\label{eq:renyi-gap}
\begin{aligned}
  J_r(u)-V_r
  ={}&\widetilde\Phi_r(\lambda_u)-V_r
  +c\log\left(1+\frac{\Lambda_u}{c}\right)\\
  &+\alpha\mathsf{D}_{\mathrm{R}}^{(1/r)}(R_u\Vert R_{\lambda_u}).
\end{aligned}
\end{equation}
The three nonnegative terms describe scalar suboptimality, loss excess above the induced threshold, and discrepancy between $R_u$ and $R_{\lambda_u}$.  Here $\lambda_u=m_u+c$ is determined by the candidate $u$ and need not solve the threshold equation.

\paragraph{A sufficient condition for uniqueness.}
Assume (15) and suppose, for every $\lambda>\ell_\star$, that $P(\ell=\lambda)=0$,
\[
  B_{s-1}(\lambda)=\int_{\{z\in\mathcal{Z}:\ell(z)<\lambda\}} (\lambda-\ell(z))^{s-1}\,P(\mathrm{d}z)<\infty,
  \qquad A_{s+1}(\lambda)B_{s-1}(\lambda)<rA_s(\lambda)^2,
\]
and differentiation under the integral for $A_s$ is valid.  Then
\[
  \frac{\mathrm{d}}{\mathrm{d}\lambda}\frac{A_{s+1}(\lambda)}{A_s(\lambda)}
  =(s+1)-s\frac{A_{s+1}(\lambda)B_{s-1}(\lambda)}{A_s(\lambda)^2}>0.
\]
Theorem~4 guarantees that an optimizer exists and that its threshold solves the threshold equation.  Strict monotonicity gives at most one solution.  Hence the global optimizer is unique.

\paragraph{Consistency of the empirical scalar calculation.}
Assume $\ell_\star> -\infty$.  Let $Z_i$ be independent with law $P$, and replace $A_{s+1}$ by its empirical average to obtain $\Phi_{r,N}$ on $\lambda>\ell_{\star,N}=\min_{i\le N}\ell(Z_i)$.  Here $N$ is the number of prior draws used to approximate the moments, and the loss function is held fixed.  Theorem~4, applied to the empirical prior $N^{-1}\sum_{i=1}^N\delta_{Z_i}$, shows that each $\Phi_{r,N}$ attains its minimum.  On every compact interval $C\subset(\ell_\star,\infty)$, the functions $(\lambda-\ell)_+^{s+1}$ have a common bounded envelope and are uniformly Lipschitz in $\lambda$.  The strong law on successively finer finite grids and the Lipschitz bound give uniform convergence of the empirical moments on $C$ almost surely.  Since $\inf_C A_{s+1}>0$, uniform convergence of $\Phi_{r,N}$ to $\Phi_r$ on $C$ follows on the same event.  Applying this argument to a countable compact exhaustion of $(\ell_\star,\infty)$ gives a single probability-one event on which the convergence holds on every compact subinterval.

To keep the empirical minimizers away from the ends of $(\ell_\star,\infty)$, note that the deterministic lower bound $\Phi_{r,N}(\lambda)\ge\lambda-c\log(\lambda-\ell_\star)$ of Proposition~\ref{prop:renyi-value}(iv) holds on each empirical domain.  Since $\ell_{\star,N}\downarrow\ell_\star$ almost surely, each fixed $\lambda_0>\ell_\star$ belongs to the empirical domain for all sufficiently large $N$.  For such $N$, $\Phi_{r,N}(\lambda_0)\to\Phi_r(\lambda_0)$ almost surely.  Consequently all empirical minimizers eventually lie in a common compact subinterval of $(\ell_\star,\infty)$.  Uniform convergence there implies that every limit point minimizes $\Phi_r$.  If $\Phi_r$ has a unique minimizer $\lambda^\star$, every sequence of empirical minimizers converges almost surely to $\lambda^\star$.

\section{Finite-support calculation for order-two R\'enyi regularization}
\label{app:renyi-finite}
\label{ex:renyi-r2-finite}

For order-two R\'enyi regularization on $\mathcal{Z}=\{1,\ldots,K\}$, put $\ell_i=\ell(i)$ and $q_i=Q(\{i\})$, and use the uniform prior $P(\{i\})=1/K$.  Then $u_i=\mathrm{d}Q/\mathrm{d}P(i)=Kq_i$, and
\[
  \int_{\mathcal{Z}} u(z)^2\,P(\mathrm{d}z)=\frac1K\sum_{i=1}^K K^2q_i^2=K\sum_{i=1}^K q_i^2.
\]
Hence the finite-dimensional objective is
\[
  F(q)=\sum_{i=1}^K q_i\ell_i+\alpha\log\left(K\sum_{i=1}^K q_i^2\right),
  \qquad q\in\Delta_K=\{q\in[0,\infty)^K:\sum_{i=1}^Kq_i=1\}.
\]
The simplex is compact, $\sum_{i=1}^K q_i^2>0$ on $\Delta_K$, and $F$ is continuous on $\Delta_K$, so a global minimizer exists.  The constant term $\alpha\log K$ affects the objective value but not the minimizing weights.

Taking $r=2$, so that $h=2$, in Theorem~4 and (16) gives
\begin{equation}
\label{eq:renyi-r2-finite}
  q_i^\star=Q^\star(\{i\})=\frac{(\lambda^\star-\ell_i)_+}{\sum_{j=1}^K(\lambda^\star-\ell_j)_+},
  \qquad
  \alpha=\frac12\frac{\sum_{j=1}^K(\lambda^\star-\ell_j)_+^2}{\sum_{j=1}^K(\lambda^\star-\ell_j)_+}.
\end{equation}
Let $A=\{i:q_i>0\}$ be the active set of a candidate minimizer with threshold $\lambda$.
Thus $A=\{i:\ell_i<\lambda\}$.  On a fixed nonempty active set $A$, the candidate weights are
\[
  q_i=\frac{\lambda-\ell_i}{\sum_{\substack{j=1\\j\in A}}^K(\lambda-\ell_j)},\quad i\in A,
  \qquad q_i=0,\quad i\notin A,
\]
and the threshold equation (16) becomes
\[
  2\alpha\sum_{\substack{i=1\\i\in A}}^K(\lambda-\ell_i)=\sum_{\substack{i=1\\i\in A}}^K(\lambda-\ell_i)^2.
\]
Put $m=|A|$, $L_A=\sum_{\substack{i=1\\i\in A}}^K\ell_i$, and $S_A=\sum_{\substack{i=1\\i\in A}}^K\ell_i^2$.  The threshold condition is then equivalent to the quadratic equation
\[
  m\lambda^2-2(L_A+\alpha m)\lambda+S_A+2\alpha L_A=0.
\]
Its candidate roots are $\lambda=L_A/m+\alpha\pm\sqrt{\alpha^2+(L_A/m)^2-S_A/m}$ when the discriminant is nonnegative.  Equivalently, real roots require the variance of the active losses, $S_A/m-(L_A/m)^2$, to be at most $\alpha^2$.
Only roots satisfying $\ell_i<\lambda$ for $i\in A$ and $\ell_i\ge\lambda$ for $i\notin A$ are consistent with the chosen active set.  If the losses are ordered from smallest to largest, every possible active set consists of the first $m$ entries of that list for some $m$, subject to the threshold inequalities.  Tied losses are included or excluded together.  By Theorem~4, every global minimizer has this threshold form and satisfies (16), so evaluating the original objective $F$ on the finitely many candidates satisfying the threshold inequalities identifies the global minimizers.  At each such candidate, $F(q)=\Phi_2(\lambda)-2\alpha+2\alpha\log(2\alpha)$ by \eqref{eq:renyi-profile-remainder}, so comparing the candidates through $F$ is the same as comparing them through $\Phi_2$.

For an active set consisting of the first $m$ entries,
\[
  A_2(\lambda)=\frac{m\lambda^2-2L_A\lambda+S_A}{K},
\]
so cumulative sums of $\ell_i$ and $\ell_i^2$ supply both the quadratic coefficients and the scalar objective in constant time per candidate.  Sorting costs $O(K\log K)$, checking all $K$ possible active-set sizes costs $O(K)$, and the selected weight vector is then constructed once in $O(K)$ time.

\section{Proof of Theorem~5}
\label{app:renyi-subunit}

Fix $0<r<1$, and let $V_r$ and $J_r$ be the value and objective in \eqref{eq:renyi-primal-main} for this order.  Write $k=r/(1-r)$ and $d=\alpha k=\alpha r/(1-r)$.  Alongside the profile $\Psi_r$ of (20), defined when $\ell_\star>-\infty$, we use its shifted version
\begin{equation}
\label{eq:renyi-subunit-shifted-profile}
 \widetilde\Psi_r(\tau)
 =\Psi_r(\tau)+d-d\log d
 =\tau+d-\alpha\log\{d^kC_k(\tau)\},
 \qquad \tau\le\ell_\star.
\end{equation}
The two profiles differ by a constant, so they have the same maximizers; the shift is chosen so that the supremum of $\widetilde\Psi_r$ is $V_r$.  The proof of Theorem~5 also determines the value of the infimum.

\begin{proposition}[Value of the infimum for $0<r<1$]
\label{prop:renyi-subunit-value}
Let $0<r<1$.  If $\ell_\star=-\infty$, then $V_r=-\infty$.  If $\ell_\star>-\infty$, then $V_r$ is finite, $\Psi_r$ is strictly concave on $(-\infty,\ell_\star)$, and
\[
 V_r=\sup_{\tau<\ell_\star}\widetilde\Psi_r(\tau).
\]
In particular, if $\ell_\star>-\infty$ and $d<\rho_r$, the infimum is finite and unattained.  If $d>\rho_r$, the supremum is attained at the scalar $\tau^\star<\ell_\star$ of Theorem~5; if $d\le\rho_r$, it is approached only as $\tau\uparrow\ell_\star$.  Whenever (18) holds, $\tau^\star$ is the unique maximizer of $\Psi_r$ over $(-\infty,\ell_\star]$, and $\int_{\mathcal Z}\ell\,\mathrm dQ^\star=\tau^\star+d$.
\end{proposition}

Thus at the boundary $d=\rho_r$ a solution exists although the supremum over $\tau<\ell_\star$ is not attained.  We prove Theorem~5 and Proposition~\ref{prop:renyi-subunit-value} together.  Suppose first that $\ell_\star=-\infty$.  Choose a bounded density $b$ supported on a positive-probability set of bounded losses.  For each $n$, choose a positive-probability set on which the loss is bounded and less than $-n$, and let $b_n$ be its conditional-prior density.  The density $u_n=(b+b_n)/2$ has finite expected loss, tending to $-\infty$, whereas
\[
 \int u_n^r\,\mathrm dP
 \ge 2^{-r}\int b^r\,\mathrm dP>0.
\]
Its divergence is therefore bounded, proving $V_r=-\infty$.

If an expected loss of $-\infty$ is admitted, the same construction also gives a density of value $-\infty$.  Choose a subsequence with $E(\ell^-u_{n_j})\ge4^j$ and form $\sum_{j\ge1}2^{-j}u_{n_j}$.  The positive-loss integral stays finite, while the negative-loss integral is infinite; every probability density has finite R\'enyi divergence for $0<r<1$.  The same mixture argument applies to the sequence constructed above for $r>1$ when $E(\ell^-)^{r/(r-1)}=\infty$, because its uniform $L^r$ bound keeps the mixture's divergence finite.

Now assume $\ell_\star>-\infty$ and write $w=\ell-\ell_\star\ge0$.  For $\tau<\ell_\star$ and $p>0$, the moments $C_p(\tau)$ are finite and positive, and monotone convergence gives $C_p(\tau)\uparrow C_p(\ell_\star)=\int_{\mathcal Z}w^{-p}\,\mathrm dP\in(0,\infty]$ as $\tau\uparrow\ell_\star$.  Differentiation under the integral on compact subintervals of $(-\infty,\ell_\star)$ gives $C_p'(\tau)=pC_{p+1}(\tau)$.  For
\[
R(\tau)=\frac{C_k(\tau)}{C_{k+1}(\tau)},
\]
Cauchy--Schwarz and the bound $\ell-\tau\ge \ell_\star-\tau$ give
\[
R'(\tau)=k-(k+1)\frac{C_k(\tau)C_{k+2}(\tau)}{C_{k+1}(\tau)^2}\le-1,
\qquad
R(\tau)\ge \ell_\star-\tau.
\]
Thus $R$ decreases strictly from $\infty$ to a finite limit $\rho_r\ge0$.  Moreover,
\[
\Psi_r'(\tau)=1-\frac{d}{R(\tau)},
\qquad
\Psi_r''(\tau)=\frac{dR'(\tau)}{R(\tau)^2}<0.
\]
Dominated convergence gives $(\ell_\star-\tau)^k C_k(\tau)\to1$ as $\tau\to-\infty$.  Hence $\log C_k(\tau)=-k\log(\ell_\star-\tau)+o(1)$, and the linear term in $\Psi_r(\tau)$ gives $\Psi_r(\tau)\to-\infty$.  If $d>\rho_r$, the unique maximizer of $\Psi_r$ solves $R(\tau)=d$, and the bound $R(\tau)\ge\ell_\star-\tau$ places it in $[\ell_\star-d,\ell_\star)$.  If $d\le\rho_r$, $\Psi_r$ is strictly increasing and its supremum is the limit at $\ell_\star$.  The same statements hold for $\widetilde\Psi_r$, which differs from $\Psi_r$ by a constant.

\paragraph{Boundary moments.}
To identify the boundary moments, suppose first that $C_{k+1}(\ell_\star)=\infty$.  For $\delta,\varepsilon>0$, splitting the integral over $\{w<\delta\}$ and its complement gives
\[
C_k(\ell_\star-\varepsilon)
\le (\delta+\varepsilon)C_{k+1}(\ell_\star-\varepsilon)+\delta^{-k}.
\]
Letting $\varepsilon\downarrow0$ and then $\delta\downarrow0$ shows that $\rho_r=0$.

If $C_{k+1}(\ell_\star)<\infty$, then $P(w=0)=0$ and $C_k(\ell_\star)<\infty$, since $w^{-k}\le1+w^{-(k+1)}$.  Monotone convergence gives
\[
\rho_r=\frac{C_k(\ell_\star)}{C_{k+1}(\ell_\star)}>0.
\]
When $d\le\rho_r$, the value of $\widetilde\Psi_r$ at $\ell_\star$ is its limit from below, by monotone convergence.  Then
\[
\sup_{\tau<\ell_\star}\widetilde\Psi_r(\tau)
=\widetilde\Psi_r(\ell_\star)=\ell_\star+d-\alpha\log\left\{d^kC_k(\ell_\star)\right\}.
\]

\paragraph{A lower bound.}
For $0<r<1$, put $\kappa=\alpha/(1-r)>0$, so that
\[
  J_r(u)=m_u-\kappa\log S_r(u),\qquad
  S_r(u)=\int_{\mathcal{Z}}u^r\,\mathrm{d}P,
  \quad m_u=\int_{\mathcal{Z}}\ell u\,\mathrm{d}P.
\]
For every feasible $u$, concavity of $x\mapsto x^r$ gives $0<S_r(u)\le1$.  The integral $S_r$ is strictly concave, and the logarithm is concave and increasing.  Thus $J_r$ is strictly convex on its finite-objective domain.  More precisely, let $v$ be a strictly positive function with $S_r(v)<\infty$.  Whenever the indicated integrals are finite, the tangent inequalities for $\log$ and $x^r$ give
\begin{equation}
\label{eq:renyi-subunit-support}
  -\kappa\log S_r(u)\ge-\kappa\log S_r(v)
   -\frac{d}{S_r(v)}\int_{\mathcal{Z}}v^{r-1}(u-v)\,\mathrm{d}P.
\end{equation}
The inequality is strict unless $u=v$ $P$-almost surely.  The reference function $v$ need not have unit mass.

For each $\tau<\ell_\star$, choose $v$ so that the linear term in \eqref{eq:renyi-subunit-support} cancels the expected loss $m_u$:
\[
  v_\tau(z)=\frac{d}{C_k(\tau)}(\ell(z)-\tau)^{-(k+1)}.
\]
Since $r(k+1)=k$, direct calculation gives
\[
  S_r(v_\tau)=d^r C_k(\tau)^{1-r},\qquad
  \frac{d\,v_\tau^{r-1}}{S_r(v_\tau)}=\ell-\tau,\qquad
  \int_{\mathcal{Z}}(\ell-\tau)v_\tau\,\mathrm{d}P=d.
\]
Applying \eqref{eq:renyi-subunit-support} with $v=v_\tau$ and using $\int_{\mathcal{Z}} u\,\mathrm{d}P=1$, we obtain, for every feasible $u$ of finite objective,
\[
\begin{aligned}
  J_r(u)
  &\ge m_u-\kappa\log S_r(v_\tau)
          -\int_{\mathcal{Z}}(\ell-\tau)(u-v_\tau)\,\mathrm{d}P\\
  &=\tau+d-\kappa\log S_r(v_\tau)
   =\tau+d-\alpha\log\{d^kC_k(\tau)\}
   =\widetilde\Psi_r(\tau).
\end{aligned}
\]
The lower bound also holds when $J_r(u)=\infty$, and equality for a feasible $u$ is possible precisely when $u=v_\tau$ $P$-almost surely.  The mass of $v_\tau$ is
\[
  \int_{\mathcal{Z}}v_\tau\,\mathrm{d}P
  =\frac{dC_{k+1}(\tau)}{C_k(\tau)}=\frac{d}{R(\tau)}.
\]
\paragraph{Attainment.}
Thus $v_\tau$ is feasible exactly when $R(\tau)=d$.  For $d>\rho_r$, the unique maximizer $\tau^\star<\ell_\star$ of $\Psi_r$ has this property, so $v_{\tau^\star}$ attains the lower bound and is the unique optimizer.  If $d=\rho_r$, both boundary moments are finite.  The same calculation applies at $\tau=\ell_\star$, where $v_{\ell_\star}$ has unit mass and is the unique optimizer; in this case $\Psi_r$ is strictly increasing on $(-\infty,\ell_\star)$ with a finite limit at $\ell_\star$, so $\tau^\star=\ell_\star$ is its unique maximizer over $(-\infty,\ell_\star]$.  In both cases $v_{\tau^\star}=(\ell-\tau^\star)^{-(k+1)}/C_{k+1}(\tau^\star)$, and the identity $\int_{\mathcal{Z}}(\ell-\tau^\star)v_{\tau^\star}\,\mathrm{d}P=d$ gives $\int_{\mathcal{Z}}\ell\,\mathrm{d}Q^\star=\tau^\star+d$.

\paragraph{Non-attainment.}
If $d<\rho_r$, the finite boundary moments and $P(\ell=\ell_\star)=0$ allow $\tau=\ell_\star$ in \eqref{eq:renyi-subunit-support}.  Put
\[
  v(z)=\frac{d}{C_k(\ell_\star)}w(z)^{-(k+1)},\qquad
  \theta=\int_{\mathcal{Z}}v\,\mathrm{d}P=\frac{d}{\rho_r}<1.
\]
The function $v$ satisfies $dv^{r-1}/S_r(v)=w$ and $\int_{\mathcal{Z}} wv\,\mathrm{d}P=d$, so the same supporting-hyperplane argument gives
\[
  J_r(u)\ge\ell_\star+d-\kappa\log S_r(v)=\widetilde\Psi_r(\ell_\star).
\]
Equality would require $u=v$ $P$-almost surely, which is impossible for a feasible density because $v$ has mass $\theta<1$.  Simply normalizing $v$ does not attain the bound: for $\bar u=v/\theta$, direct substitution gives $J_r(\bar u)-\widetilde\Psi_r(\ell_\star)=d\{\theta^{-1}-1+\log\theta\}>0$ because $0<\theta<1$.

To show that this lower bound is the infimum, let $E_\varepsilon=\{z\in\mathcal{Z}:0<w(z)<\varepsilon\}$.  Since $P(w=0)=0$ and $\ell_\star$ is the essential infimum, $P(E_\varepsilon)>0$ and $P(E_\varepsilon)\to0$.  We supply the missing mass $1-\theta$ on sets whose losses approach $\ell_\star$.  The densities
\[
  u_\varepsilon(z)=v(z)+\frac{1-\theta}{P(E_\varepsilon)}\mathbf{1}_{E_\varepsilon}(z)
\]
have unit mass and satisfy $\ell_\star+d\le m_{u_\varepsilon}\le\ell_\star+d+(1-\theta)\varepsilon$, so $m_{u_\varepsilon}\to\ell_\star+d$.  Subadditivity of $x\mapsto x^r$ gives
\[
  0\le S_r(u_\varepsilon)-S_r(v)
  \le(1-\theta)^rP(E_\varepsilon)^{1-r}\longrightarrow0.
\]
Because $S_r(v)=d^rC_k(\ell_\star)^{1-r}$, it follows that $J_r(u_\varepsilon)\to\widetilde\Psi_r(\ell_\star)$.  This proves $V_r=\sup_{\tau<\ell_\star}\widetilde\Psi_r(\tau)$ and non-attainment.  The minimizing densities $u_\varepsilon$ place the fixed mass $1-\theta$ on sets of vanishing $P$-probability on which the loss converges uniformly to $\ell_\star$.

\begin{remark}[Another route to the lower bound]
\label{rem:renyi-birrell}
The linear variational representation of \citet[Thm.~2.1]{birrell2023function} gives another route to the bound $J_r(u)\ge\widetilde\Psi_r(\tau)$.  For bounded losses, take their order to be $1-r$.  Their divergence $R_{1-r}(P\Vert Q)$ equals $\mathsf{D}_{\mathrm{R}}^{(r)}(Q\Vert P)/r$, and the choice $g=(\tau-\ell)/(\alpha r)$ gives $\widetilde\Psi_r(\tau)$.  The substitution extends by truncation to the lower-bounded losses used here.
\end{remark}

\bibliographystyle{ba2}
\bibliography{refs}